\documentclass[a4paper,11pt]{amsart}
\usepackage{amssymb,amsthm,bbm,epic,eepic,graphics,amsbsy,url}
\usepackage{amssymb,amsthm}
\usepackage{xypic}
\usepackage[applemac]{inputenc}
\usepackage{longtable}
\usepackage{array}
\newcommand{\la}{\lambda}

\def\dd{\mathrm{d}}
\DeclareMathOperator{\rk}{\mathrm{rk}}

\DeclareMathOperator{\Irr}{\mathrm{Irr}}
\def\ad{\mathrm{ad}}
\def\Ad{\mathrm{Ad}}

\def\Hom{\mathrm{Hom}}

\def\End{\mathrm{End}}
\def\Ker{\mathrm{Ker}\,}
\def\Res{\mathrm{Res}}

\def\Aut{\mathrm{Aut}}

\def\diag{\mathrm{diag}}

\def\C{\ensuremath{\mathbbm{C}}}
\def\Q{\mathbbm{Q}}
\def\Z{\mathbbm{Z}}

\def\R{\mathbbm{R}}
\def\bar{\overline}
\def\gl{\mathfrak{gl}}
\def\sl{\mathfrak{sl}}
\def\so{\mathfrak{so}}

\def\osp{\mathfrak{osp}}

\def\kk{\mathbbm{k}}
\def\kt{\mathbbm{k}^{\times}}

\def\om{\omega}

\def\eps{\epsilon}
\def\g{\mathfrak{g}}
\def\h{\mathfrak{h}}

\def\un{\mathbbm{1}}
\def\onto{\twoheadrightarrow}
\def\into{\hookrightarrow}
\def\ii{\mathrm{i}}
\def\Gal{\mathrm{Gal}}
\def\Irr{\mathrm{Irr}}
\def\GL{\mathrm{GL}}

\def\Sp{\mathrm{Sp}}

\def\bla{\boldsymbol{\lambda}}
\def\bula{\underline{\boldsymbol{\lambda}}}
\def\ula{\underline{\lambda}}
\def\bmu{\boldsymbol{\mu}}
\def\bumu{\underline{\boldsymbol{\mu}}}
\def\umu{\underline{\mu}}

\def\AAA{\mathbbm{A}}

\def\bT{\mathbf{T}}
\def\LRef{\ensuremath{\Lambda\mathrm{Ref}}}
\def\QRef{\ensuremath{\mathrm{QRef}}}
\def\Ref{\ensuremath{\mathrm{Ref}}}

\def\Id{\mathrm{Id}}
\newtheorem{theo}{Theorem}[section]
\newtheorem{theointr}{Theorem}
\newtheorem{prop}[theo]{Proposition}

\newtheorem{lemma}[theo]{Lemma}

\newtheorem{cor}[theo]{Corollary}

\newtheorem{conj}{Conjecture}
\def\mod{\ \mathrm{mod}\ }

\makeatletter
\DeclareRobustCommand\widecheck[1]{{\mathpalette\@widecheck{#1}}}
\def\@widecheck#1#2{%
   \setbox\z@\hbox{\m@th$#1#2$}%
   \setbox\tw@\hbox{\m@th$#1%
      \widehat{%
         \vrule\@width\z@\@height\ht\z@
         \vrule\@height\z@\@width\wd\z@}$}%
   \dp\tw@-\ht\z@
   \@tempdima\ht\z@ \advance\@tempdima2\ht\tw@ \divide\@tempdima\thr@@
   \setbox\tw@\hbox{%
      \raise\@tempdima\hbox{\scalebox{1}[-1]{\lower\@tempdima\box\tw@}}}%
   {\ooalign{\box\tw@ \cr \box\z@}}}
\makeatother

\title{{\bf Infinitesimal Hecke algebras III}}
\author{Ivan Marin}
\date{December 20, 2010}
\begin{document}

\maketitle

\bigskip
\begin{center}
Institut de Math\'ematiques de Jussieu \\
Universit\'e Paris 7 \\
175 rue du Chevaleret \\
F-75013 Paris
\end{center}
\bigskip

\bigskip

\bigskip

\noindent {\bf Abstract.} We define Lie subalgebras of the group
algebra of a finite \emph{pseudo-}reflection group that are involved
in the definition of the Cherednik KZ-systems, and determine
their structure. We provide applications for computing the Zariski
closure of the image of generalized (pure) braid group $B$ inside
the representations of the corresponding Hecke algebras. We also
get unitarizability results for the representations of $B$
originating from Hecke algebras for suitable parameters,
and relate our Lie algebras with the topological closure of $B$ in these
compact cases.

\medskip

%20C99,20F36.

\def\RR{\mathcal{R}}
\def\SS{\mathcal{S}}
\def\HH{\mathcal{H}}
\def\XX{\mathrm{X}}

\section{Introduction}

Let $V$ be a finite-dimensional complex vector space,
and $W < \GL(V)$ a finite subgroup generated by pseudo-reflections,
namely invertible endormorphisms of $V$ which fix some hyperplane.
Such a group is called a (pseudo-)reflection group. We denote
by $\RR$ the set of (pseudo-)reflections of $W$.

Among them are the 2-reflection groups, namely when $\forall s \in \RR \ s^2 = 1$.
Examples of 2-reflection groups include the finite Coxeter groups. For a
2-reflection group, we introduced in \cite{HECKINF} the Lie subalgebra of the group
algebra $\C W$ generated by $\RR$, and call it the infinitesimal
Hecke algebra. This is a reductive Lie algebra that appears naturally in the
monodromy constructions of braid groups representations, also known as KZ systems.
This Lie algebra was decomposed in \cite{LIETRANSP} for the case of the
symmetric group, and in \cite{IH2} for the general case of 2-reflection
groups. We proved in \cite{IH2} that it can be identified with
the Lie algebra of an interesting algebraic group,
namely the Zariski closure of the image of the (generalized) braid group
inside the generic Hecke algebra associated to $W$. It also admits
a compact form that is relevant to the topological closure
of the same group, under some conditions on the
parameters involved.

The present paper is a continuation of \cite{IH2}. Here
we consider the general case of (pseudo-)reflection
groups.

There are several natural generalizations of the
infinitesimal Hecke algebras. The first one is
the Lie subalgebra $\HH$ of $\C W$ generated by the set $\RR$ of
all reflections. A second one is the Lie subalgebra $\HH_s$ generated
by a subset $\SS \subset \RR$ which admit the following
properties : it is invariant under $W$-conjugation,
and $s \mapsto \Ker(s-1)$ induces a bijection from $\SS$
to the reflecting hyperplanes.

For a conjugacy class $c \subset \SS$, we let $e_c$ denote the order of the
elements of $c$, and subdivide $\SS = \SS_+ \sqcup \SS_0$
with $e_c = 2$ for $s \in \SS_0$ and $e_c > 2$ for $s \in \SS_+$.
We let $\mathcal{C}_+$ denote the set of conjugacy classes in $\SS_+$,
and $\mathbbm{A}(W) = \prod_{c \in \mathcal{C}_+} \kk^{e_C -1}$
for a fixed field $\kk$ of characteristic 0.
A typical element of $\mathbbm{A}(W)$ is denoted
$\bula = (\ula^c)_{c \in \mathcal{C}_+}$,
with $\ula^c = (\la^c_1,\dots,\la^c_{e_c-1})$.
We thus get a family of generalizations, that contains $\HH_s$, by letting
$\HH(\bula)$ denote the Lie subalgebra of $\kk W$ generated by
$\SS_0$ and the $\la_1^c s + \dots + \la^c_{e_c-1} s^{e_c-1}$,
for $s \in c \subset \SS_+$.

From now on we assume that $W$ is not a 2-reflection group,
this case having been dealt with in \cite{IH2}, and also
that it is irreducible.

We introduce
subsets $\overline{\QRef}, \mathcal{E}, \mathcal{F}$ of
the set $\Irr(W)$ of irreducible representations of $W$, as well as
an equivalence relation $\approx$ in $\Irr(W)$.
The central result is a structure theorem for $\HH$, where we denote
$V_{\rho}$ the underlying vector space of $\rho \in \Irr(W)$.

\begin{theointr} $\HH$ is a reductive Lie algebra whose center
has dimension the cardinality $|\RR/W|$ of conjugacy classes of reflections.
Its semisimple part is
$$
\HH' = 
\left(\bigoplus_{\rho \in \overline{\QRef}} \sl(V_{\rho}) \right)
\oplus \left(\bigoplus_{\rho \in \mathcal{E}/\approx} \sl(V_{\rho}) \right)
\oplus \left(\bigoplus_{\rho \in \mathcal{F}/\approx} \osp(V_{\rho}) \right)
$$
\end{theointr}

We then
prove that there exists a dense open subset of $\AAA(W)$
over which $\HH(\bula)$ has the same semisimple part as $\HH$.

\begin{theointr} The Lie algebras $\HH$ and $\HH(\bula)$ are
reductive for all values of $\bula$. There exists an hyperplane complement
$\AAA(W)^{\times} \subset \AAA(W)$ such that, for
all $\bula \in \AAA(W)^{\times}$, $\HH(\bula)$
has a center of dimension $|\SS/W|$, and
semisimple part $\HH(\bula)' = \HH'$
\end{theointr}

\medskip

We then establish the connection between $\HH(\bula)$
and the Zariski closure of the (pure) braid groups inside
the Hecke algebra associated to $W$, and finally investigate
several interesting cases for $\bula \not\in \AAA(W)^{\times}$ :
\begin{enumerate}
\item for the Lie algebra $\HH_s$, whose parameter $\bula$
does not always belong to $\AAA(W)^{\times}$ ;
\item  for the exceptional groups $G_4,G_{25}$, which are
connected to the usual braid groups, we investigate the structure
of $\HH(\bula)$ for arbitrary $\bula \in \AAA(W)$ ;
\item when $\bula$ satisfies $\la^c_i = \la^c_j$, which corresponds
to the `spetsial' Hecke algebra of \cite{BMM}, we investigate
$\HH(\bula)$ for the groups $G(d,1,r)$.
\end{enumerate}

\def\HHU{\HH(\mathcal{U})}
\def\HHTU{\HH(\tilde{\mathcal{U}})}

It has been noted in \cite{IH2} that
the infinitesimal Hecke algebras for 2-reflection groups
are actually generated by any set of reflections that
generate the groups, e.g. a set of simple reflections
for $W$ a finite Coxeter groups. Other generalizations
to consider are then, for $\mathcal{U}$ a subset of $\SS$
that generates the group and $\tilde{\mathcal{U}} = \{ s^k \ | \ 
s^k \neq 1, s \in \mathcal{U} \} \subset \RR$,
the Lie algebras $\HH(\mathcal{U})$ and $\HH(\tilde{\mathcal{U}})$
generated by these subsets. Both Lie algebras are
reductive. We have $\HHU \subset \HH_s$,
and in general $\HH_s' \neq \HH'$, so $\HHTU$ is more likely
to be considered as a natural generalization than $\HHU$.
Our last result
on this topic is the following

\begin{theointr} If $\mathcal{U} \subset \mathcal{S}$ generates
$W$, then $\HHU$ and $\HHTU$ are reductive, and $\HHTU' = \HH'$.
\end{theointr}

In order to get this result, we study the following situation.
For a finite group $G$, there are natural Lie-theoretic
endomorphisms of $\End(\C G)$ associated to $g \in G$,
namely $\ad(g) : x \mapsto gx -xg$ and $\Ad(g) : x \mapsto g x g^{-1}$.
The technical result proved in the final section is the following.

\begin{theointr} Let $k \geq 1$. There exists a rational polynomial $P
\in \Q[X_1,\dots,X_{k-1}]$ such that,
for every finite group $G$ and $g \in G$ of order $k$, then $\Ad(g) = P(\ad(g),\ad(g^2),\dots,\ad(g^{k-1}))$.
Moreover $P$ can be taken in $\Q[X_1]$ if $k$ is odd or equal to 2.
\end{theointr}
%that $\Ad(g)$ can be expressed
%as a rational polynomial in the $\ad(g^k)$ that depends only on
%the order $n$ of $g$, and as a rational polynomial in $\ad(g)$
%for $n$ an odd integer or $n =2$.

The applications of these structure results on the Zariski
closures of braid groups in their
corresponding Hecke algebras, are discussed
in \S 5 (see theorem \ref{theozarheck}). In \S 6 we prove that, for convenient
values of the parameters and, if $W$ is irreducible, with the possible exception
of a finite number of types, the representation of the Hecke
algebra are unitarizable as representations of the corresponding
braid group $B$ (theorem \ref{theounit} and its corollary), and we relate
the Lie algebra of the topological closure of $B$ in these compact cases
with the Lie algebras introduced here.

\medskip

\noindent {\bf Acknowledgements.} I thank D. Juteau, J.-F. Planchat
and K. Sorlin for useful discussions as well as E. Opdam for comments
on a first version of this work.

\section{Infinitesimal Hecke algebras}

Let $W$ be a (pseudo-)reflection group, $\RR$ its set of
(pseudo-)reflections and $\mathcal{C}$ the set of
conjugacy classes of reflection hyperplanes. 
We consider $\mathcal{S} \subset \mathcal{R}$
such that $\mathcal{S}$ is stabilized by conjugation
under $W$ and generates $W$ as a group. We will show below that a
consequence of this assumption and of Stanley's theorem (see \cite{STANLEY}, theorem 3.1)
is that
$$
\forall s \in \RR \ \exists s_0 \in \SS \ \exists r \geq 1 \ s = s_0^r.
$$

Let $\kk$ be a field of characteristic 0. We introduce the
Lie subalgebra $\mathcal{H}(\mathcal{S})$ of $\kk W$ generated
by $\mathcal{S}$ and, for $c$ a conjugacy class of $W$,
we let $T_c = \sum_{w \in c} w \in \kk W$.

We recall from \cite{IH2} the following.
\begin{prop} (see \cite{IH2} prop. 2.2)
%\begin{enumerate}
%\item 
The Lie algebra $\mathcal{H}(\mathcal{S})$
is reductive, and a basis of its center is given by the $T_c$ for $c \subset
\mathcal{S}$. Every irreducible representation of $W$ induces an
irreducible representation of $\HH(\SS)$. The derived Lie algebra $\mathcal{H}(\mathcal{S})'$
is generated by the $s' = s- (T_c)/(\# c)$ for $s \in c$ a conjugacy
class included in $\mathcal{S}$.
%\end{enumerate}
\end{prop}

%For a group $G$ and $g \in G$ of finite order, we consider the endomorphisms
%of $\End(\C G)$ defined by $\ad (g) : x \mapsto [g,x]$
%and $\Ad(g) : x \mapsto gxg^{-1}$.

%\begin{prop} Let $k \in \Z_{>1}$.
%There exists $P_k \in \Q[X_1,\dots,X_{k-1}]$ such that, for all groups $G$
%and $g \in G$ of order $k$, $\Ad(g) = P(\ad(g),
%\ad(g^2),\dots,\ad(g^{k-1}))$. Moreover, if $k$ is odd or $k = 2$,
%then $P_k$ can be taken in $\Q[X_1]$.
%\end{prop}
%In particular, one can take $P_2 = \dots$, $P_3 = \dots$,
%$P_4 = \dots$, $P_5 = \dots $.

We say that a reflection is \emph{primitive} if it generates the fixer
of its reflecting hyperplane.

%\begin{cor} If $\SS_0 \subset \SS$ is a set of primitive reflections
%in $\SS$ with at least one representative for each conjugacy class
%in $\SS$, then $\mathcal{H}_{\SS}$ is generated by $\SS_0$ as a Lie algebra.
%\end{cor}

%For clarity we let $\mathcal{H} = \mathcal{H}(\mathcal{S})$.
We assume that $\kk$ is a field such that every ordinary representation
of $W$ is realizable over $\kk$ (this means that $\kk$ contains
the so-called field of definition of $W$, by \cite{benard,bessis}).
Without loss of generality,
we can assume that $\kk$ is a number field, with $\kk \subset \C$.
Since $\mathcal{H}(\SS) \subset \kk W$, every such representation 
$\rho$ of $W$ induces a representation $\rho_{\mathcal{H}(\SS)}$ of $\mathcal{H}(\SS)$.
Similarly, its restriction to $\mathcal{H}(\SS)'$ is denoted $\rho_{\mathcal{H}(\SS)'}$.

We let $\Irr(\rho)$ denote the set of irreducible representations of $W$,
and define for $\rho \in \Irr(W)$ the set
$$
\mathrm{X}_{\mathcal{S}}(\rho) =
\{ \eta \in \Hom(W,\kk^{\times}) \ | \ \forall s \in \mathcal{S} \  \ 
\eta(s) \neq 1 \Rightarrow \rho(s) \in \kk^{\times} \}.
$$

\begin{prop}
\begin{enumerate}
\item For any $s \in \RR$ there exists $s_0 \in \SS$ and $r \geq 1$ such
that $s = s_0^r$.
\item 
For any $\rho \in \Irr(W)$, $\XX(\rho) = \XX_{\SS}(\rho)$
does not depend on $\SS$. %\footnote{diary 21 mars 09}
\end{enumerate}
\end{prop}
\begin{proof}
For $C$ of conjugacy class of hyperplanes in $W$ we denote $e_C$ the
order of the fixer of an hyperplane in $C$. By Stanley's theorem
there exists an isomorphism $\Hom(W,\C^{\times}) \simeq \prod_{C} \Z/e_C \Z$
where a primitive pseudo-reflection around $H$ in $C$ is mapped
to a generator of $\Z/e_C \Z$. It follows that every pseudo-reflection
$s \in \RR$ around such an $H$ is mapped to some $x \in \Z/e_C \Z$.
Since $\SS$ generates $W$, there exists $s_0 \in \SS$ and $r \geq 1$
with $s_0^r$ having the same image as $s$ in $\Z/e_C \Z$. Since
$\SS$ is stable by conjugation we can assume that $s_0$ fixes
the same hyperplane as $s$ hence $s_0^r = s$. This proves (1).
We let $\SS_1,\SS_2 \subset \RR$, $\rho \in \Irr(W)$ and $\eta \in \XX_{\SS_1}(\rho)$.
Let $s \in \SS_2$ such that $\rho(s) \not\in \kt$. By (1) we have $r \geq 1$
and $s_0 \in \SS_1$ such that $s = s_0^r$, so $\rho(s) \not\in \kt$
implies $\rho(s_0) \not\in \kt$, hence $\eta(s_0) = 1$ and $\eta(s) = 1$.
It follows that $\XX_{\SS_1}(\rho) \subset \XX_{\SS_2}(\rho)$
and $\XX_{\SS_1}(\rho) = \XX_{\SS_2}(\rho)$ whence (2).
\end{proof}

For $\mathcal{U} \subset \RR$ not necessarily stable under
conjugation, we may consider the Lie subalgebra $\HHU$ of $\kk W$ generated
by $\mathcal{U}$. Letting $\tilde{\mathcal{U}} = \{ s^k \in \RR \ | \ s \in \mathcal{U}, k \geq 1 \}$,
we have the following, which shows that the stability under
$W$ plays a role mainly for the center.

\begin{prop} If $\mathcal{U}$ generates $W$, then $\HHU$ and
$\HHTU$ are reductive, and $\HHTU' = \HH'$.
\end{prop}
\begin{proof}
We decompose $\kk W = Z(\kk W) \oplus (\kk W)'$ as a Lie algebra.
Then $Z(\kk W)$ and $(\kk W)'$ are ideals of $\kk W$ as an
associative algebra, and there exists a $W$-equivariant
idempotent $p = \kk W \onto Z(\kk W)$. Since $\mathcal{U}$
and $\tilde{\mathcal{U}}$ generate $\HHU$ and $\HHTU$, respectively,
both Lie algebras are reductive (see \cite{IH2} prop. 2.2). Now
$\RR = \{ w s w^{-1} \ | \ s \in \tilde{\mathcal{U}}, w \in W \}$
by Stanley theorem, hence $\HH$ is generated by the $\HH(w \tilde{\mathcal{U}}
w^{-1})$ for $w \in W$. Since $\tilde{\mathcal{U}}$ generates
$W$, $Z(\HHTU) \subset Z(\kk W)$, hence $p(Z(\HHTU)) = \{ 0 \}$.
From the reductiveness of $\HHTU$ we get
$\HHTU = Z(\HHTU) \oplus \HHTU'$, hence
$p(\HHTU) = p(\HHTU') = \HHTU'$, as $\HHTU' \subset (\kk W)'$.
More generally,
$$p(\HH(w\tilde{\mathcal{U}}w^{-1})) = \HH(w\tilde{\mathcal{U}}w^{-1})
 = w \HHTU' w^{-1} = \Ad(w)(\HHTU').
$$
Let now $s \in \mathcal{U}$. The endomorphism $\ad(s^k) \in \End(\kk W)$
stabilizes $\HHTU$ hence $\HHTU'$, as $\HHTU = Z(\HHTU) \oplus \HHTU'$
and $s^k \in \tilde{\mathcal{U}}$. We will prove later (theorem \ref{polynome})
that $\Ad(s)$ can be written as a rational polynomial
in $\ad(s),\dots,\ad(s^k),\dots$, hence $
s \HHTU' s^{-1} %= \HH(s \tilde{\mathcal{U}}s^{-1})'
 \subset \HHTU'$. Since $\Ad(s)$ is invertible, it follows that
$s \HHTU' s^{-1} = \HHTU$. Since $\mathcal{U}$ generates $W$
we get $w \HHTU' w^{-1} = \HHTU'$ for all $w \in W$.
Then $p(\HH(w \tilde{\mathcal{U}} w^{-1})) = \HHTU'$ hence $\HH'$
is generated by $\HHTU'$, that is $\HH' = \HHTU'$. 
\end{proof}

For $\rho \in \Irr(W)$ we let $V_{\rho}$ denote its underlying vector space.
We define the following subsets of $\Irr(W)$.
$$
\begin{array}{l}
\Ref = \{ \rho \in \Irr(W) \ | \ \dim \rho \geq 2 \mbox{ and } \forall s \in \SS \ \rho(s) \neq 1 \Rightarrow
\rho(s) \mbox{ is a reflection } \} \\
\QRef = \{ \eta \otimes \rho \ | \ \rho \in \Ref, \eta \in \Hom(W,\kt) \} \\
\LRef = \{ \eta \otimes \Lambda^k \rho \ | \ \rho \in \Ref, \eta \in \Hom(W,\kt),k \geq 0 \} \\
\end{array}
$$

The statements and proofs of the propositions 2.13 and 2.15 in \cite{IH2} admit a
natural generalization.
%\footnote{23 octobre 2008, page 46.}
\begin{prop} {\ } \\
\begin{enumerate}
\item If $\rho \in \Ref(W)$ and $\eta \in \Hom(W,\kt)$, then
$\Lambda^k \left((\rho \otimes \eta)_{\HH(\SS)'}\right) = (\eta \otimes \Lambda^k \rho)_{\HH(\SS)'}$.
\item If $\rho \in \QRef(W)$ then $\rho(\HH({\SS})') = \sl(V_{\rho})$.
\end{enumerate}
\end{prop}
\begin{proof}
Let $\rho \in \Ref(W)$, and $s \in \SS$. If $\rho(s) = 1$ then
$(\eta \otimes \Lambda^k \rho)(s) = \eta(s)$ and $\Lambda^k(\rho \otimes \eta)(s)
= k \eta(s) = (k-1)\eta(s) + (\eta \otimes \Lambda^k \rho)(s)$. Otherwise,
there exists a basis $e_1,\dots,e_n$ of $V_{\rho}$ such that $s.e_1 = \zeta e_1$
and $s.e_i = e_i$ if $i \neq 1$ for $\zeta$ some nontrivial root of 1.
Taking for basis of $\Lambda^k V_{\rho}$ the basis $e_{I} = e_{i_1} \wedge \dots
\wedge e_{i_k}$ for $I = \{ i_1,\dots,i_k\} \subset [1..n]$ of cardinality $k$
with $i_1<\dots<i_k$ we get that $(\eta \otimes \Lambda^k \rho)(s)$
maps $e_I$ to $\eta(s) e_I$ if $1 \not\in I$ and to $\eta(s) \zeta e_I$
if $1 \in I$. Similarly, $\Lambda^k((\rho \otimes \eta)_{\HH(\SS)})$
maps $e_I$ to $k \eta(s)e_I$ if $1 \not\in I$ and to $\eta(s)(k-1 + \zeta)e_I$
otherwise. It follows that $\Lambda^k((\rho \otimes \eta)_{\HH(\SS)}(s) =
\eta(s)(k-1)\mathrm{Id} + (\eta \otimes \Lambda^k \rho)(s)$ for
all $s \in \SS$, which easily implies
$\Lambda^k((\rho \otimes \eta)_{\HH(\SS)'} =
(\eta \otimes \Lambda^k \rho)_{\HH(\SS)'}$ and proves (1).
It is clearly enough to prove (2) for $\rho \in \Ref$. Since $\rho(W)$
is an irreducible reflection group, we can also assume that $\rho$ is faithful.
We proceed by induction on the rank $n = \dim V_{\rho}$ of $W$
for $n \geq 1$, the case $n = 1$ being trivial, so we assume $n \geq 2$.
Let $W_0 \subset W$ be a maximal parabolic subgroup which acts
irreducibly on some hyperplane $H$ of $V_{\rho}$ (see \cite{IH2} lemma 2.17).
The image $\g$ of $\HH(\SS)'$ in $\sl(V_{\rho})$ contains $\sl(H)$, hence a Cartan
subalgebra of rank $n-2$. Since $n \geq 2$, there exists $s \in \SS \cap W_0$
with $\rho(s) \not\in \kt$. Let $\zeta \neq 1$ with $\zeta \in \mathrm{Sp}(s)$.
We denote $c_0$ the conjugacy class of $s$ in $W_0$ and $c$ its class in $W$,
$T_0 = \sum_{g \in c_0} g$, $T = \sum_{g \in c} g$. Since $T$ is central in
$\HH(\SS)$ and $\rho$ is irreducible, $\rho(T)$ is a scalar determinated
by its trace $(\# c)(n-1 + \zeta)$. Similarly, $\rho(T_0)$
has for trace $(\# c_0)(n-1+\zeta)$, acts on $H$ by $(\# c_0)(n-2+\zeta)/(n-1)$
and on its orthogonal supplement by $\# c_0$. 
Letting $X = (\# c)T_0 - (\# c_0)T$ we get that
$x = \rho(X)$ has zero trace,
belongs to $\sl(V_{\rho}) \cap \rho(\HH(\SS)) = \rho(\HH(\SS)')$
(since $\HH(\SS)$ is reductive), is semisimple and centralizes the Cartan of $\sl(H)$.
It follows that the semisimple Lie algebra $\g$
contains a Cartan subalgebra of rank $n-1$. Since $\sl_n$ contains no proper
root system of rank $n-1$ (see \cite{IH2}, lemma 2.16) it follows that $\g = \sl(V_{\rho})$,
and this proves (2) by induction.
\end{proof}

We now focus on two special cases, fixing $\mathcal{S}$
to be minimal, for instance by letting
$\SS$ being the set of distinguished pseudo-reflections.
We denote $\HH_s = \HH(\SS)$, the letter `s' being understood
as the `special' infinitesimal Hecke algebra.
The second one is for $\SS$ maximal; we let $\mathcal{H} = \HH(\RR)$
and call it the \emph{ambient} infinitesimal Hecke algebra. 
We denote $\mu_{\infty}(\C)$ the group of roots of 1.

\begin{lemma} \label{lemcercle}
Let $\mathcal{X} \subset \mu_{\infty}(\C)$ and
$\om \in \C^{\times}$ such that $\om + \mathcal{X} \subset \mu_{\infty}$.
Then $|\mathcal{X}| \leq 2$ and, if $|\mathcal{X}| = 2$
with $\mathcal{X} = \{ \alpha, \beta \}$, then $\om = -\alpha - \beta$.
\end{lemma}
\begin{proof}
This is a consequence of the fact that, if $\alpha,\beta \in \C$
with $\alpha \neq \beta$, then the equation $|\alpha+z|^2 = |\beta+ z|^2=1$
has at most two solutions ; for $|\alpha| = |\beta | = 1$
these solutions are $0$ and $-\alpha-\beta$. Since $\omega \neq 0$,
$\alpha,\beta,\gamma \in \mathcal{X}$ would imply $\alpha + \beta = \alpha
+ \gamma$ hence $\beta = \gamma$ and $|\mathcal{X}| \leq 2$.
\end{proof}
 
\begin{prop} \label{propHHisom}
\begin{enumerate}
%\item If $\rho \in \Irr(W)$, then $\rho_{\mathcal{H}}$ is irreducible.
\item If $\rho^1, \rho^2 \in \Irr(W)$ then $\rho^1 \simeq \rho^2
\Leftrightarrow (\rho^1)_{\mathcal{H}} \simeq (\rho^2)_{\mathcal{H}}$.
\item If $\rho^1,\rho^2 \in \Irr(W)$ with $\dim \rho^i > 1$,
then  $\rho^1_{\mathcal{H}'} \simeq \rho^2_{\mathcal{H}'}$ iff
$\rho^2 \simeq \rho^1 \otimes \eta$ for some $\eta \in \mathrm{X}(\rho)$.
%\footnote{Pour ça, voir la note
%du 23 octo obre 2008, page 45}
%\item For any $\rho \in \Irr(W)$, $\XX(\rho) = \XX_{\SS}(\rho)$
%does not depend on $\SS$. %\footnote{diary 21 mars 09}
\end{enumerate}
\end{prop}
\begin{proof}
(1) is a direct consequence of the fact that $\mathcal{H}$
and $W$ are both generated by $\RR$. For $(2)$, we assume that
$(\rho^1)_{\HH'}$ and $(\rho^2)_{\HH'}$ are irreducible and identify
the underlying vector spaces of $\rho^1$, $\rho^2$. Denoting it by $V$,
there exists $P \in \GL(V)$ with $P \rho^2(s') P^{-1} = \rho^1(s')$ for
each $s \in \RR$, with $s' = s-(1/N) T_s$, $N$ the cardinality of
the conjugacy class $c$ of $s$ and $T_s = \sum_{g \in c} g$. Since
$\rho^1,\rho^2$ are irreducibles we have $T_s \in \kk$. We let
$\om_s = (\rho^2(T_s) - \rho^1(T_s))/N$, so that $P \rho^2(s) P^{-1} = 
\rho^1(s) + \om_s$. Raising the previous equation to the square we get
$P \rho^2(s^2) P^{-1} = \rho^1(s^2) + 2 \om_s \rho^1(s) + \om_s^2$.
We apply this to some $s \in \SS$. If $s^2 = 1$,
then we have $1 = 1 + 2 \om_s \rho^1(s) + \om_s^2$ hence $\om_s = 0$
or $\rho^1(s) = -\om_s/2 \in \kt$. Otherwise, $s^2 \in \RR$
hence $P \rho^2(s^2) P^{-1} = \rho^1(s^2) + \om_{s^2}$
and $\om_{s^2} =   2 \om_s \rho^1(s) + \om_s^2$,
so also in this case $\om_s = 0$ or $\rho^1(s) \in \kt$.

For $C$ the set of conjugacy classes of hyperplanes we define
$J \subset C$ by $c \in J$ iff the corresponding $s \in \SS$ satisfy $\om_s \neq 0$.
For such an $s$ we define $\eta(s) \in \kt$ by $\rho^2(s) = \rho^1(s) \eta$,
and define $\eta(s) = 1$ otherwise. Stanley's theorem extends this
formula to a character $\eta \in \Hom(W,\kt)$ such that $\rho^2 \simeq
\rho^1 \otimes \eta$ with $\eta \in \XX(\rho^1)$. Conversely,
if $\rho^2 = \rho^1 \otimes \eta$ with $\eta \in \XX(\rho^1)$
and $s \in \RR$, then either $\rho^2(s) = \rho^1(s)$, or $\rho^2(s),
\rho^1(s) \in \kt$ hence $\rho^1(s) \eta(s) = \rho^1(s) + \om$ for
some $\om \in \C^{\times}$, which implies $\rho^2(s') = \rho^1(s')$.
This concludes the proof of (2). 
\end{proof}

\begin{prop} \label{propduauxHH} Let $\rho \in \Irr(W)$.
There exists $\rho^1 \in \Irr(W)$ such that $\rho^1_{\HH_s'} \simeq (\rho_{\HH_s'})^*$
if and only if, for any $s \in \SS$, % with $s^2 \neq 1$,
%\begin{enumerate}
%\item
$\rho(s)$ has at most two eigenvalues.
%\item The map $s \mapsto \alpha+\beta - ^t\rho(s)$
%\end{enumerate}
In that
case, $\rho^1 \simeq \rho^* \otimes \chi$ with $\chi \in \Hom(W,\kt)$
and $\chi(s)$ for $s \in \SS$ is equal to the product of the eigenvalues of $\rho(s)$
whenever $\rho(s)$ has two distinct ones ; moreover,
one also has $\rho^1_{\HH'} \simeq (\rho^*)_{\HH'}$.
%\footnote{diary 26 octobre 2008, page 51, pour la première version (fausse) de la preuve}
\end{prop}
\begin{proof}
We assume $\rho^1_{\HH_s'} \simeq (\rho_{\HH_s'})^*$, and identify
the underlying vector space with a common $\kk^n$. Let $s \in \SS$
and $X = \rho(s)$, $Y = \rho^1(s)$. We have $P \in \GL_n(\kk)$
independent from $s$ with $P \rho^1(s')P^{-1} = - ^t \rho(s')$ i.e.
$PYP^{-1} = - ^t X + \om_s$
with $\om_s = (1/\#c)(\rho^1(T) + \rho(T))$, $c$ being the conjugacy
class of $s$, $T = \sum_{g \in c} g$.
%Since $\rho(T) = (1/n) \tr \rho(T)
%= (\# c/n) \tr X$ and $\rho^1(T) = (\# c /n) \tr Y$ we have
%$\om_s = (1/n)\tr(X+Y)$.

We assume $\om_s \neq 0$ and let $n$ denote the order of $s$. From
$P Y P^{-1} = - ^t X + \om_s$
and $\om_s - \mathrm{Sp} (X) \subset \mu_{\infty}$ we get from
lemma \ref{lemcercle} that $| \Sp X| \leq 2$ and, if
$\Sp(X) = \{ \alpha,\beta \}$ with $\alpha \neq \beta$, we have
$\om_s = \alpha + \beta$ hence, since $X$ is semisimple,
$\om_s - X = \beta \alpha X^{-1}$, whence $PYP^{-1} = u_s ^t X^{-1}$
for some $u_s \in \mu_n$; similarly,
if $|\Sp(X)| = 1$, then $PYP^{-1} = u_s ^t X^{-1}$ for some
$u_s \in \mu_n$. Then $s^k \mapsto u_s^k$ defines a character of
the subgroup generated by $s$.

%We now choose $s \in \SS$.
%If $s^2 = 1$ we have $X^2 = Y^2 = 1$ hence $1 = 1 - 2 \om_s \, ^t X + \om_s^2$.
%This implies that either $\om_s = 0$, or $X = (\om_s/2) \in \kk$
%hence $Y = -(\om_s/2)$, and in both cases $P \rho^1(s) P^{-1} = - ^t\rho(s)
%= - ^t \rho(s^{-1})$.

%If $s^2 \neq 1$, then $s^2 \in \RR$ so similarly $PY^2P^{-1} = - ^tX^2 + \om_{s^2}$.
%Raising the first equation to the square we get $PY^2P^{-1} = \, ^t X^2 - 2
%\om_s \, ^t X + \om_s^2$ hence $2 \, ^t X^2= 2 \om_s \, ^t X + \om_{s^2} -
%\om_s^2$. This implies that $X$ has at most two eigenvalues $\alpha,\beta
%\in \kt$. Since $X$ is semisimple, if $\alpha=\beta$ then $X =\alpha
%\in \kt$ and $Y= - ^t X + \om_s = \om_s - \alpha$. Since $Y ,X \neq 0$
%there exists $u \in \kt$ with $P Y P^{-1} = u_s ^t X^{-1}$, which
%defines a character $s^k \mapsto u^k$ of the subgroup generated by $s$.

%Now assume $\alpha \neq \beta$. Since the eigenvalues $\om_s - \alpha, \om_s- \beta$ of $Y$ are
%roots of 1, we have $|\om_s - \alpha| = |\om_s - \beta| = 1$,
%which implies $\om_s = \alpha + \beta$ or $\om_s = 0$.
%Note that the solution $\om_s=0$ is illicit when $s$ has odd order $d$,
%because in that case, since $(-\alpha)^d \neq 1$, $\om_s - \alpha = -\alpha$
%could not be an eigenvalue for $Y$.

%Moreover,
%from $PYP^{-1} = - ^t X + \om_s$ we get that the multiplicity
%of $\alpha$ in $X$ equals the multiplicity of $\om_s - \alpha$ in 
%$Y$. Assume $\om_s = \alpha+\beta$. Putting $X$, which is semisimple,
%in diagonal form, we get that $- ^t X + \alpha+\beta = \alpha\beta ^t X^{-1}$.
Let now $\chi \in \Hom(W,\kt)$ defined by
$\chi(s) = -1$ if $\om_s = 0$ or $s^2 = 1$, $\chi(s) = u_s$ for $\rho(s) \in \kt$
with $u$ defined as above, and otherwise
$\chi(s)$ equal to the product of the eigenvalues of $\rho(s)$.
We then have $\rho^1 \simeq \chi \otimes \rho^*$ and conversely,
if  $\rho^1 = \chi \otimes \rho^*$ with such a $\chi$, then it is easily
checked that $(\rho^1)_{\HH'} \simeq (\rho^*)_{\HH'}$.
\end{proof}

Note that any two $\chi$ as in the above statement are deduced from
each other by tensoring by some $\eta \in \XX(\rho^1)$.

We define an equivalence relation $\approx$ on $\Irr(W)$
generated by $\rho^1 \approx \rho^2$ for $\rho^2 = \rho^1 \otimes \eta$,
$\eta \in \XX(\rho^1)$, and $\rho^2 = (\rho^1)^* \otimes \chi$ for
a $\chi$ as in proposition \ref{propduauxHH}. We let $\overline{\QRef}$
a set of representatives in $\QRef$ of $\QRef/\approx$, and subdivide
$\Irr(W) \setminus \LRef = \mathcal{E} \sqcup \mathcal{F}$ with $\mathcal{F}$
the $\rho$ such that $\rho_{\HH'} \simeq (\rho_{\HH}')^*$. We identify
$\mathcal{E}/\approx$ and $\mathcal{F}/\approx$
with some subset of representatives. Finally, if $\rho_{\HH'} \simeq \rho_{\HH'}^*$,
there exists a nondegenerate bilinear form over $V_{\rho}$ preserved
by $\rho(\HH')$. We say that $\rho \in \Irr(W)$
is of \emph{orthogonal type} if this bilinear form is symmetric,
and of \emph{symplectic type} otherwise, namely if this bilinear
form is skew-symmetric. In both cases,
we denote the corresponding orthogonal or symplectic
algebra by $\osp(V_{\rho})$.

We now define the following Lie algebra
$$
\mathcal{M} = \left(\bigoplus_{\rho \in \overline{\QRef}} \sl(V_{\rho}) \right)
\oplus \left(\bigoplus_{\rho \in \mathcal{E}/\approx} \sl(V_{\rho}) \right)
\oplus \left(\bigoplus_{\rho \in \mathcal{F}/\approx} \osp(V_{\rho}) \right)
$$
and embed $\mathcal{M}$ inside $\kk W$ in the obvious way.
For instance, if $\rho_0$ is a representative of
some class in $\mathcal{E}$, then the $\sl(V_{\rho})$
factor in $\mathcal{M}$ is identified with a diagonal factor
$\bigoplus_{\rho \approx \rho_0} \sl(V_{\rho})$
using the isomorphisms $\rho_{\HH'} \simeq \rho^0_{\HH'}$
or $\rho_{\HH'}^* \simeq (\rho^0_{\HH'})^*$.

By the above propositions, we have $\HH' \subset \mathcal{M}$.
In the next section we will prove the following theorem.

\begin{theo} \label{theostruct} If $W$ is irreducible and $W \neq H_4$, then
$\HH' = \mathcal{M}$.
\end{theo}

The proof of the following lemma is postponed to the next section.

\begin{lemma} \label{lemrhoexcept} Let $W$ be irreducible and not a 2-reflection
group. Let $\rho \not\in \LRef$ with $\rho_{\HH'} \simeq (\rho_{\HH'})^*$.
Then $\dim \rho$ is an even integer and,
if $\dim \rho \in \{ 4, 6 , 8 \}$,
then $\rho$ is of symplectic type.
\end{lemma}

%\begin{lemma} \label{lemrhoexcept} Let $W$ be irreducible and not a 2-reflection
%group. Then, for all $\rho \not\in \LRef$,
%if $\dim \rho = 4$ with $\rho_{\HH'} \simeq (\rho_{\HH'})^*$ then $\rho(\HH') \not\simeq \so_4$ ;
%if $\dim \rho = 8$ with $\rho_{\HH'} \simeq (\rho_{\HH'})^*$ then $\rho(\HH') \not\simeq \so_8$ ;
%if $\dim \rho = 6$ with $\rho_{\HH'} \simeq (\rho_{\HH'})^*$ and $\rho(\HH') \simeq \so_6$, then
%$\rho \in \LRef$ ; if $\dim \rho = 2N+1>1$ with $\rho_{\HH'} \simeq (\rho_{\HH'})^*$, then $\rho(\HH') \not\simeq \so_{2N+1}$.
%\footnote{preuve a verifier, et a etendre a $\HH_s'$.}
%\end{lemma}

Let $\g \subset \mathcal{M}(\rho)$ a semisimple Lie subalgebra of $\mathcal{M}$.
For all $\rho \in \Irr(W)$, we denote $\mathcal{M}(\rho)$
and $\g(\rho)$
the simple Lie ideal of $\mathcal{M}$ and $\g$, respectively,
that corresponds to $\rho$, i.e. the orthogonal w.r.t. the Killing
form of the kernel of $\rho$ for the corresponding Lie algebras.
We have clearly $\mathcal{M}(\rho) \simeq \sl(V_{\rho})$ for $\rho \in \mathcal{E}$
and $\mathcal{M}(\rho) \simeq \osp(V_{\rho})$ for $\rho \in \mathcal{F}$.
By the above lemma, all these Lie ideals are simple, as the non-simple
case $\so_4 \simeq \sl_2 \times \sl_2$ does not arise.

\begin{lemma} \label{lemtheo} Let $W$ be irreducible and not a 2-reflection
group. If we have
\begin{enumerate}
\item For all $\rho \in \Irr(W)$, $\rho(\mathcal{M}) \simeq \rho(\g)$
%\item For all $\rho\not\in \LRef$, $\rho_{\g}^* \simeq \rho_{\g}$ iff $\rho \in \mathcal{F}$.
\item $\rho^1_{\g} \simeq \rho^2_{\g}$ iff $\rho^1_{\mathcal{M}}
\simeq \rho^2_{\mathcal{M}}$
\item $\rho^1_{\g} \simeq (\rho^2_{\g})^*$ iff $\rho^1_{\mathcal{M}}
\simeq (\rho^2_{\mathcal{M}})^*$
\end{enumerate}
then $\g = \mathcal{M}$.
\end{lemma}
\begin{proof}
Let $\mathcal{X} = \overline{\QRef} \cup \mathcal{E}/\approx
\cup \mathcal{F} / \approx$, so that $\mathcal{M} = \bigoplus_{\rho
\in \mathcal{X}} \mathcal{M}(\rho)$.

Each of the $\rho \in \Irr(W)$ provides a Lie ideal
$\g(\rho)$ of $\g$, namely the orthogonal of
$\Ker \rho_{\g}$ w.r.t. the Killing form of $\g$,
which is isomorphic to $\rho(\g)$, which is isomorphic
to $\rho(\mathcal{M})$ by (1) and simple by the above lemma,
as the case $\so_4$ is excluded.
Moreover, the type of $\rho(\g)$, for $\g \in \mathcal{X}$,
determines $\dim \rho$ by the above lemma, as the exceptional
isomorphisms $\so_3 \simeq \sl_2$ and $\so_6 \simeq \sl_4$
are excluded by the lemma for $\rho \not\in \LRef$,
and $\LRef \cap \mathcal{X} = \overline{\QRef}$.

Let $\rho^1,\rho^2 \in \mathcal{X}$ such that
$\g(\rho^1) = \g(\rho^2)$. This is possible if and only if
$\rho^1$ and $\rho^2$ have the same dimension, say $N$.
But the simple Lie ideals involved here have at most 2
irreducible representations of dimension $N$,
as the exceptional case $\so_8$ is excluded. Moreover,
when there are two of them, one is the dual of the other.
Thus $\rho^1_{\g} \simeq \rho^2_{\g}$ or $\rho^1_{\g} \simeq (\rho^2_{\g})^*$.
If $\rho^1_{\g} \simeq \rho^2_{\g}$ or $\rho^1_{\g} \simeq (\rho^2_{\g})^*$,
by (2) and (3) we have $\rho^1 = \rho^2$, as $\rho^1,\rho^2 \in
\mathcal{X}$. Then
$\g = \bigoplus_{\rho \in \mathcal{X}} \g(\rho) \simeq
\bigoplus_{\rho \in \mathcal{X}} \rho(\mathcal{M}) \simeq \mathcal{M}$,
hence $\dim \g = \dim \mathcal{M}$ and $\g = \mathcal{M}$.

%If $\rho^1_{\g} \simeq \rho^2_{\g}$, by (3) we have
%$\rho^1_{\mathcal{M}} \simeq \rho^2_{\mathcal{M}}$
%hence $\rho^1 = \rho^2$ as $\rho^1,\rho^2 \in \mathcal{X}$.
%We thus can assume $\rho^1_{\g} \simeq (\rho^2_{\g})^*$
%with $\rho^1_{\g} \not\simeq (\rho^1_{\g})^*$.
%Then $\rho^1_{\mathcal{M}} \not\simeq (\rho^1_{\mathcal{M}})^*$
%, $\rho^2_{\mathcal{M}} \not\simeq (\rho^2_{\mathcal{M}})^*$
%hence $\rho^1,\rho^2 \in \mathcal{E} \cup \overline{\QRef}$.

%If $\rho^1,\rho^2 \in \QRef$, this implies

%This implies $\rho^1, \rho^2 \not\in \QRef$,
%except for $N = 2$.

% by (1) and the above lemma.
% Since, for all $\rho \in \mathcal{X}$,
%we have $\rho(\g) \simeq \rho(\mathcal{M})$, then
%$\g = \mathcal{M}$ if and only if, for all $\rho^1,\rho^2 \in \mathcal{X}$,
%$\g(\rho^1) = \g(\rho^2) \Leftrightarrow \mathcal{M}(\rho^1) = 
%\mathcal{M}(\rho^2)$.

%Moreover, these Lie ideals admit at most 2 irreducible
%representations with dimension $\dim \rho$, and,
%for $\rho \not\in \LRef$, exactly two if and only
%if $(\rho_{\g})^* \not\simeq \rho_{\g}$, that is $\rho \in \mathcal{F}$
%by (2).

\end{proof}

\bigskip
\noindent {\bf Example for $G_{25}$.}
We first describe $\Irr(W)$ for $W$ of type $G_{25}$. Let $j = \exp(2 \ii \pi/3)$,
and choose $s \in \SS$. It has order 3, and there is only one conjugacy
class of hyperplanes for $W$, so there are exactly 3 one-dimensional
$S_{\alpha} : W \to \kk^{\times}$ and $\XX(\rho) = \{ S_1 \}$
for every $\rho \in \Irr(W)$ with $\dim \rho \geq 2$. The 24 irreducible
representations are described in the following table.

$$
\begin{array}{|llll|}
\hline
\mbox{Dim.} & \mbox{Name} & \mbox{Parameters} & \mathrm{Sp}(s) \\
 1 & S_{\alpha} & \alpha \in \mu_3 & (\alpha) \\
 2 & U_{\alpha,\beta} & \{ \alpha,\beta \} \subset \mu_3, \alpha \neq \beta
& (\alpha,\beta) \\
3 & U'_{\alpha,\beta} & (\alpha,\beta) \in (\mu_3)^2, \alpha \neq \beta &
(\alpha,\alpha,\beta) \\
3 & V & & (1,j,j^2) \\ 
6 & V_{\alpha,\beta} & (\alpha,\beta) \in (\mu^3)^2, \alpha \neq \beta & (\alpha,\alpha,\alpha,\beta,\beta,\gamma) \\
8 & W_{\alpha} & \alpha \in \mu_3 & (\alpha,\alpha,\alpha,\alpha,\beta,\beta,\gamma,\gamma ) \\
9 & X & & (\alpha,\alpha,\alpha,\beta,\beta,\beta,\gamma,\gamma,\gamma) \\
9 & X^* & & (\alpha,\alpha,\alpha,\beta,\beta,\beta,\gamma,\gamma,\gamma) \\
\hline
\end{array}
$$
From these datas, we get that $\LRef = \{ S_{\alpha}, U_{\alpha,\beta},
U'_{\alpha,\beta} \ | \ \alpha,\beta \in \mu_3 \}$,
and $\QRef = \{ U_{\alpha,\beta }, U'_{\alpha,\beta} \ | \ \alpha,\beta \in
\mu_3 \}$. Moreover $U'_{\alpha,\beta} \approx U'_{\beta,\alpha}$
is the only nontrivial identification in $\QRef/\approx$.
We will prove later on (see proposition \ref{propduauxexcept}) that the
identifications are the same for $\HH'_s$ and $\HH'$,
since every pseudo-reflection has odd order.
It follows that the reflection ideal is made out of simple
ideals $\sl_2$ and $\sl_3$ for each pair $\{ \alpha,\beta \}$
and is isomorphic to $\sl_2^3 \times \sl_3^3$. Outside $\LRef$,
each $N$-dimensional representation corresponds to a distinct ideal $\sl_N$,
$N \in \{ 6,8,9 \}$, so
$$
\HH'_s \simeq 
\mathcal{H}' \simeq \sl_2^3 \times \sl_3^4 \times \sl_6^6 \times \sl_8^3\times \sl_9^2
$$

\section{Proof of the structure theorem}

We prove theorem \ref{theostruct} separately for the infinite series
and for the exceptional groups. Since this theorem
is known to hold for $W$ a 2-reflection group by \cite{IH2},
we assume that $W$ is not such a group.

In order to prove the theorem, by lemma \ref{lemtheo}, we only need
to prove lemma \ref{lemrhoexcept} as well as the following one.

\begin{lemma} \label{lemplein}
Let $W$ be an irreducible reflection group which is not a 2-reflection
group. For any $\rho \not\in \LRef$, $\rho(\HH_s') = \osp(V_{\rho})$
if $\rho_{\HH'} \simeq (\rho_{\HH'})^*$ and
$\rho(\HH_s') = \sl(V_{\rho})$ otherwise.
\end{lemma}

We prove both lemmas in the two cases of the general series and
the exceptional cases.

\subsection{The infinite series $G(de,e,r)$}
\label{proofstructgdeer}

We let $W = G(de,e,r)$ and $W_0 = G(de,de,r)$. Since $W_0$
is a 2-reflection group, the structure of its infinitesimal Hecke algebre
$\mathcal{H}_0 \subset \mathcal{H}_s$ is known by \cite{IH2}. Recall
that $W_0$ is an index $d$ subgroup of $W$, and that $W = W_0 \rtimes < t >$
where $t \in \SS$ is the pseudo-reflection $\diag(\zeta,1,\dots,1)$ for $\zeta$ some
primitive $d$-th root of 1. We let $\alpha : W \to \mu_d$ with kernel $W_0$,
$\alpha(t) = \zeta$. When needed, we will use the standard labelling
by $de$-tuples of partitions of total size $r$ of the irreducible
representations of $G(de,1,r)$, as in e.g. \cite{ARIKI,ARIKIKOIKE,MARINMICHEL}.

%, and denote $\HH_0 \subset \HH_s$ the infinitesimal Hecke
%algebra of $W_0$.

We extract from \cite{IH2} the following criterium. For $\h$ a semisimple
Lie algebra, we let $\rk \h$ denote its semisimple rank.

\begin{lemma} \label{critsimple}
Let $\h \subset \g \subset \sl_N$ be semisimple Lie algebras with $\g$ acting
irreducibly on $\kk^N$. If $\rk \h > N/2$, or if $\rk \h = N/2$ with $\kk^N$
not selfdual as a $\g$-module, then $\g = \sl_N$
\end{lemma}
\begin{proof}
If $\rk \g > N/2$ this is \cite{IH2} lemma 3.1. If $\rk \g = N/2$, then
by \cite{IH2} lemma 3.3, since $N < (N/2+1)^2$ for
all $N$, we get that $\g$ is simple. Then \cite{IH2} lemma 3.4
shows that the only possibility is that $\g = \sl_N$.
\end{proof}

Let $\rho \in \Irr(W)\setminus \LRef$. If $\rho(t) = u \in \kt$,
then $u = \zeta^{-k}$ for some $k$ and $\rho \otimes \alpha^k(W) = \rho(W_0)$
so the result follows from \cite{IH2}, as $(\rho \otimes \alpha^k)(\HH_s') =
\rho(\HH_0')$. From now on we assume
that $\rho(t) \not\in \kt$,
and we decompose $\Res_{W_0} \rho = \rho_1 + \dots + \rho_m$ in irreducible
components. By Clifford theory, $\rho_i \not\simeq \rho_j$ for $i \neq j$.
We denote $\eps$ the sign character on $W_0$.
%, and assume $\forall i \ \rho_i \not\in \LRef$.
Since the $\rho_i$ are deduced from each other through conjugation by $t$,
the condition $\forall i \ \rho_i \not\in \LRef(W_0)$ is equivalent to $\rho_1 \not\in \LRef(W_0)$.
We note that we can choose an ordering on $\rho_1,\dots,\rho_m$
such that $V_{\rho} = V_{\rho_1} \oplus \dots \oplus V_{\rho_m}$
with $V_{\rho_{b+k}} = \rho(t^k)V_{\rho_b}$ and $\rho_{k+1} \simeq
\rho_k \circ \Ad(t)$. A consequence is that $p = |\Sp \rho(t)| \geq m$.
Recall that, if $p \geq 3$, then $\rho_{\HH_s'}$ is not
selfdual. This is thus the case if $m \geq 3$.

\subsubsection{The case $m = 1$, $\rho_1 \not\in \LRef$}

If $\rho_1 \not\simeq \rho_1^* \otimes \eps$, then by \cite{IH2}
we already have $\rho(\HH_s') = \sl(V_{\rho})$, so we can assume
$\rho_1 \simeq \rho_1^* \otimes \eps$, and let $\h = \rho(\HH_0')$.
By \cite{IH2} this is a simple Lie algebra $\osp_0(V_{\rho})$ of rank $(\dim \rho)/2$
that preserves some nondegenerate bilinear form over $V_{\rho}$
and acts irreducibly on $V_{\rho}$.

If $(\rho_{\HH'})^* \not\simeq \rho_{\HH'}$
then by lemma \ref{critsimple}
we get $\g = \sl(V_{\rho})$, so we can assume that
 $(\rho_{\HH'})^* \simeq \rho_{\HH'}$ (in particular,
$\rho(t)$ admits exactly two eigenvalues $u,v$),
hence $\g \subset \osp(V_{\rho})$.
Since $\h$ acts irreducibly on $V_{\rho}$, it
can preserve only one such form (up to scalar), so
from $\h \subset \g \subset \osp(V_{\rho})$ and
$\h \subset \osp_0(V_{\rho})$
we get $\osp(V_{\rho}) = \osp_0(V_{\rho})$ and $\h = \g = \osp(V_{\rho})$.

\subsubsection{The case $m \geq 2$, $\rho_1 \not\in \LRef$, $\rho_{\HH_s'} \not
\simeq (\rho_{\HH_s'})^*$}

We need to show that $\g = \sl(V_{\rho})$.

First assume $\rho_1^* \otimes \eps \simeq \rho_1$.
Since $\rho_1 \not\in \LRef$ we have $\dim \rho_1 \geq 3$
(see \cite{IH2} lemma 2.18) hence $\dim \rho \geq 6$.
This implies $r \geq 3$,
as it is easily checked that the irreducible representations
of $G(de,e,2)$ have dimension at most 2. It follows
that $W_0$ has a single class of reflections. Since $W_0 \neq H_4$, the
$\rho_i$ then correspond to distinct ideals of rank
$(\dim \rho_i)/2$ and $\g$ contains
a semisimple Lie algebra of rank $m(N/2m) = N/2$. Then lemma
\ref{critsimple} implies $\g = \sl(V_{\rho})$.

Now assume $\rho_1^* \otimes \eps \not\simeq \rho_1$. If $\forall i,j \ \ 
\rho_i \not\simeq \rho_j^* \otimes \eps$, then the rank of $\h$ is
$m (N/m-1) = N-m > N/2$ iff $\dim \rho_i = N/m > 2$, which holds
true since $\rho_i \not\in \LRef$.

Otherwise, the map $\rho_i \mapsto \rho_i^* \otimes \eps$ induces
a permutation of the $\rho_i$, as $(\rho_1 \circ \Ad(t^i))^* \otimes \eps
= (\rho_1^* \otimes \eps)\circ \Ad(t^i)$,
hence $m = 2q$ for some $q \geq 1$, and $\h$ has rank $h = q(N/m-1) = N/2 - m/2$.
We have $N < (h+1)^2$ iff $N^2 - 2(m-4)N + (m-2)^2 > 0$. This trinomial
in $N$ has for reduced discriminant $-2(2m-6)<0$ for $m \geq 4$.
Checking separately the case $m = 2$, we get that
$N<(h+1)^2$. It follows that
$\g$ is simple by \cite{IH2} lemma 3.3 (I). Moreover,
since $N/m = \dim \rho_i > 2$
we have $h= N/2- m/2 > N/4$ hence $\rk \g > N/4$. 
Moreover $N = m \dim \rho_1$
with $\dim \rho_1 \geq 3$ and $m \geq 2$ hence $N \geq 8$.
Since $V_{\rho}$ is not selfdual as a $\g$-module and $N\geq 8$ is even,
by \cite{IH2} lemma 3.4 we get that the only possibility for
$\g \neq \sl_N$
implies $N = 10$ and $\g \simeq \sl_5$, or 
$N=16$ and $\rk \g = 5$. The latter case is excluded, since
it implies $\dim \rho_i = 4$, $m = 4$
and $\rk \g \geq (N-m)/2 = 6$ ; the former is excluded because
then $\dim \rho_i = 5$ and
$\g \simeq \sl_5 \simeq  \rho(\HH_0')$ implies $\g
= \rho(\HH_0')$, which contradicts the
fact that $\rho_{\g}$ is irreducible whereas $\rho_{\HH_0}'$
is not. It follows that $\g = \sl(V_{\rho})$ in this case.

\subsubsection{The case $\rho_1 \not\in \LRef, m \geq 2$ and $\rho_{\HH_s'} 
\simeq (\rho_{\HH_s'})^*$}
\label{soussection313}

This implies $m=2,p=2$. We have $\g \subset \osp(V_{\rho})$
and want to show $\g = \osp(V_{\rho})$.
%If $\rho_1 \simeq \rho_1^* \otimes \eps$, then we get as before
%$\rk \h \geq N/2$.

By Clifford theory we have that $d = \# \alpha(W)$ even, so $d = 2 q$,
that $\rho \otimes \alpha^q \simeq \rho$, and that $\rho_1 \circ \Ad(t) = \rho_2$.
Since $\alpha^q(t) = -1$, we get from $\alpha^q \otimes \rho \simeq
\rho$ and $p = 2$ that $\rho(t)$ has for eigenvalues $u,-u$
with the same multiplicities. Letting $u = \zeta^{-k}$ we get
that $\rho' = \rho \otimes \alpha^k$ has the same restriction
to $W_0$ than $\rho$ with $\rho'(t)^2 = 1$. Hence $\rho'$ 
factorizes through the classical morphism $G(de,e,r) \onto
G(de,de/2,r)=W_1$ that preserves $\mathfrak{S}_r$ and maps
$t$ to $t$. Since $W_1$ is a 2-reflection group, 

%factorizes
%through\footnote{il faut decidemment clarifier ca, j'y comprends rien
%a chaque fois que je relis.} a representation of $W_1 = \Ker \alpha_q$, and
%$\g = \rho(\HH_1') = \rho'(\HH_1')$. 
the conclusion that $\g = \rho(\HH_s')$
thus follows from \cite{IH2}, as $\rho_{\HH_s}$ is selfdual
iff $\rho' \simeq (\rho')^*\otimes \eps$ for $\eps$ the sign character of
$W_1$, and $\rho \not\in \LRef 
\Rightarrow  \rho' \not\in \LRef(W_1)$.

\subsubsection{The case $\rho_1 \in \LRef(W_0)$}

The list of representations in $\LRef(W_0)$ can be found in \cite{IH2}.
We check on this list that the only possibilities for $\rho(t)$
not to be a scalar (in which case we would have $\rho \in \LRef$)
are the following two, with $\dim \rho = 6$.

The former one is when $r = 4$, $de$ is even, $m=2$, $\rho_1$ has dimension $3$ and, up to tensoring
by some power of $\alpha$, $\rho$ is the restriction to $W$ of a representation
$(\la,\emptyset,\dots,\la,\emptyset,\dots)$ of $G(de,1,4)$, with
$\la \in \{ [2],[1,1] \}$. But then $\rho(t)^2 = 1$ and $\g$ coincides
with the image of (the semisimple part of) the infinitesimal Hecke algebra
of $G(de,de/2,4)$, so the conclusion follows from \cite{IH2}.

The latter is when $r = 3$, $m=3$, $\dim \rho_1 = 2$ and,
up to tensoring by some power of $\alpha$, $\rho$ is the restriction to $W$ of the
representation $([1],\emptyset,\dots,[1],\emptyset,\dots,[1],\emptyset,\dots,)$ of
$G(de,1,3)$. It is easily checked (e.g. through the character table of $G(3,3,3)$) that $\rho_i \simeq \rho_i^* \otimes \eps$,
hence the $\rho_i$ correspond to three distinct ideals of $\h$
and $\h$ has rank 3. Then $\rk \h \geq 6/2$, and $\rho_{\HH'_s}$
is not selfdual, as $m > 2$ and by the argument in the previous section.
By lemma \ref{critsimple} it follows that $\g = \sl(V_{\rho})$,
and this concludes the proof for the $G(de,e,r)$.

\subsubsection{Proof of lemma \ref{lemrhoexcept}}

We let $\rho \not\in \LRef$ with $\rho_{\HH'} \simeq (\rho_{\HH'})^*$
and denote $\rho_0$ the restriction of $\rho$ to $W_0$.
Assume first $\dim \rho = 4$.
If $\rho_0$ is irreducible, then
$\rho_0 \simeq \rho_0^* \otimes \eps$, where
$\eps$ denotes the sign character of $W_0$. Then $\rho$ is of orthogonal
type if and only if $\eps \into S^2 \rho_0$, which is possible
only for $W_0$ of type $F_4$ by \cite{IH2} lemma 7.3. Since
$W_0 = G(de,de,r)$ is not of this type this excludes this case.
If $\rho_0$ is not irreducible, then its irreducible
components have dimension 2 or 1, and in particular belong to $\LRef(W_0)$.
We saw above that this situation does not occur.

Now assume $\dim \rho = 8$.
%Since $\rho \not\in \LRef$,
As before,
the irreducible components of $\rho_0$ do not belong to $\LRef(W_0)$,
and this excludes the case where $\rho_0$ admits 4
irreducible components. The case of $\rho_0$ irreducible (then $\eps \into
S^2 \rho_0$) is
excluded by \cite{IH2} proposition 7.6.
If $\rho_0$ admits 2 irreducible components $\rho_1,\rho_2$ of dimension 4,
then, since $(\rho_0)_{\HH'}$ would preserve a nondegenerate
symmetric bilinear form, so would either $(\rho_1)_{\HH_0'}$ or $(\rho_2)_{\HH_0'}$.
This is excluded by \cite{IH2} lemma 7.3, which concludes the case $\dim \rho = 8$.

Finally, assume that $\dim \rho = 6$. If $\rho_0$ is irreducible,
we are done by \cite{IH2} lemma 7.4. Otherwise, it admits
either 2 or 3 irreducible components. The case of 3 irreducible
components is excluded by $(\rho_{\HH'})^* \simeq \rho_{\HH'}$,
hence $\rho_0$ is the sum of 2 irreducible
3-dimensional components. Then $d$ is even and, up to tensoring by some
multiplicative character, $\rho$ corresponds to a
multipartition of the form $([2],\dots,[2],\dots)$
that factorize through $W_1 = G(de,de/2,r)$. This
situation is then dealt with again by \cite{IH2} lemma 7.4.

Now assume $\dim \rho = 2N+1>1$. Then $\rho$ is of symmetric type if
$\rho(\HH') \simeq \so_{2N+1}$. Let $\rho_1$ denote
an irreducible component of $\rho_0$. We have $\dim \rho_1 > 1$, and also
$\rho_1(W_0) \not\subset \kk^{\times}$, otherwise
$\rho(W)$ is abelian and $\dim \rho = 1$. It follows that
$W_0$ act by the sign character $\eps$ on this form, hence this form
admits involutive skewisometries afforded by the 2-reflections in $W_0$.
This implies that this form is hyperbolic (see \cite{IH2}, lemma 2.5), hence $2N+1$ is
even, a contradiction that concludes the proof of lemma \ref{lemrhoexcept}
for the general series.

\subsection{Exceptional groups}

Among the 34 exceptional groups in the Shephard-Todd classification,
15 of them are 2-reflection groups, and were dealt with in \cite{IH2}.
The others mostly have rank 2, plus 3 groups of higher rank, namely $G_{25}$,
$G_{26}$ and $G_{32}$.

For a given $\rho \in \Irr(W)$, given an explicit matrix model $\rho : W \to \GL(V_{\rho})$,
the dimension of $\rho(\HH_s')$ can be computed by the following algorithm :
start from the $\rho(s)$ for $s \in \SS$, extract a basis for the spanned
subspace in $\gl(V_{\rho})$, add to this subspace the images of this basis under
the $\ad(\rho(s))$ for $s \in \SS$, and then iterate the process until the
dimension stops increasing ; this gives the dimension $d$ of $\rho(\HH_s)$,
and $\rho(\HH_s')$ has dimension $d$ or $d-1$ if one of the $\rho(s)$ for
$s \in \SS$ has nonzero trace. Of course this dimension is the same
for $\rho$ and for $\rho \otimes \eta$ if $\eta \in \Hom(W,\kt)$,
which reduces sometimes drastically (notably for $W$ of type $G_{19}$)
the number of representations to check.

This algorithm is tractable provided that $\dim \rho$ is reasonably
small, and that explicit matrix models are available. These two
conditions are satisfied by all these exceptional groups,
except for $G_{32}$.

For $W \neq G_{32}$, the irreducible representations have dimension
at most 9, explicit matrix models were computed by various authors, and are now
included in the CHEVIE package for the GAP3 software, which is available at
\url{http://www.math.jussieu.fr/~jmichel/}. We check using the above algorithm that $\dim \rho(\HH_s') =
\dim(\rho)^2 - 1$, that is $\rho(\HH_s') = \sl(V_{\rho})$,
for all $\rho \not\in \LRef$ (in particular this shows that $\rho_{\HH'}
\not\simeq (\rho_{\HH'})^*$ for all these representations).
This proves lemmas \ref{lemrhoexcept} and \ref{lemplein} in these cases,
and the structure theorem by lemma \ref{lemtheo}.

The only group $W$ remaining to be dealt with is $G_{32}$. In that case
we use a parabolic subgroup $W_0$ of type $G_{25}$ and, assuming the
theorem proved in type $G_{25}$, we consider, for each
$\rho \in \Irr(W)$, the (reductive) Lie algebra generated by the $\rho(s)$,
$s \in \SS \cap W_0$, and let $\h(\rho)$ denote its semisimple part
(that is, its intersection with $\sl(V_{\rho})$). By the analysis carried out
above of the infinitesimal Hecke algebra of type $G_{25}$, we know
to which simple Lie ideal corresponds each irreducible representation
of $W_0$. It follows that the rank of $\h$ can be deduced
from the knowledge of the induction table from $W_0$ to $W$, which is
known and included in CHEVIE. Using it, we get that
$\h(\rho)$ has rank greater than $\dim \rho/2$ for all $\rho \not\in \LRef$,
that is $\rho(\HH'_s) = \sl(V_{\rho})$ by lemma \ref{critsimple}.
This completes the proofs of lemmas \ref{lemrhoexcept} and \ref{lemplein},
and of the structure theorem.

\section{Generic Hecke algebras}

We let $\mathcal{C}_+$ denote the set of conjugacy
classes $c$ of hyperplanes with $e_c > 2$,
and $\mathbbm{A}(W) = \prod_{c \in \mathcal{C}_+} \kk^{e_C -1}$.
A typical element of $\mathbbm{A}(W)$ is denoted $\bula = (\ula^c)_{c \in \mathcal{C}_+}$,
with $\ula^c = (\la^c_1,\dots,\la^c_{e_c-1})$. We denote
$\SS_+ = \{ s \in \SS \ | \ s^2 \neq 1 \}$,
$\SS_0 = \{ s \in \SS \ | \ s^2 = 1 \}$, and let
$\HH(\bula)$ denote the Lie subalgebra of $\kk W$ generated by
$\SS_0$ and the $\la_1^c s + \dots + \la^c_{e_c-1} s^{e_c-1}$
for $s \in \SS_+$ with reflecting hyperplane in $c \in \mathcal{C}_+$.

\subsection{Preliminaries}

Let $\zeta \in \kk$ a primitive $n$-th root of 1. Note that, if $n$ is
the order of some pseudo-reflection of $W$, then $\kk$ contains
$\mu_n(\C)$ (e.g. because the defining representation is realizable over $\kk$).

In $\kk^n$ we define $v_i = (1,\zeta^i,(\zeta^i)^2,\dots,(\zeta^i)^{n-1})$
for $i \in [0,n-1]$. Since the corresponding (Vandermonde) determinant
is invertible, these elements form a basis of $\kk^n$. It follows that,
for $(i,j) \neq (k,l)$ with $i \neq j$ or $k \neq l$, $H_{i,j,k,l} = \{ \ula \in \kk^n \ | \ 
<\ula | v_i - v_j - v_k + v_l> = 0 \}$
is an hyperplane of $\kk^n$.

\begin{prop} Let $P \in A = \kk[X]/(X^n -1)$ and $\zeta \in \kk$ a primitive
$n$-th root of 1. Then
\begin{enumerate}
\item $P$ generates $A$ as a unital algebra if and only if,
for all $r,s \in [0,n-1]$, $r \neq s \Rightarrow P(\zeta^r) \neq P(\zeta^s)$.
\item the unital subalgebra of $A \otimes_{\kk} A = \kk[X,Y]/(X^n -1,Y^n -1)$
generated by $P(X) + P(Y)$ contains $X+Y$ if and only if,
for all $i,j,k,l$, $\zeta^i + \zeta^j \neq \zeta^k + \zeta^l \Rightarrow
P(\zeta^i) + P(\zeta^j) \neq P(\zeta^k) + P(\zeta^l)$.
\item letting $P = P_{\ula} = \la_0 + \la_1 X + \dots + \la_{n-1} X^{n-1}$,
for $\ula \not\in H_{i,j,k,l}$,
$P(\zeta^i) - P(\zeta^j) = P(\zeta^k) - P(\zeta^l)$
implies $(i,j) = (k,l)$, or $i=j$ and $k=l$.
\end{enumerate}
\end{prop}
\begin{proof} We consider $\pi_i : A \onto \kk[X]/(X - \zeta^i) = \kk$
that maps $X$ to $\zeta^i$. Then $\pi = \bigoplus_{i=0}^{n-1} \pi_i : A \to \kk^n$
is a ring isomorphism, and $\pi(P)$ generates $\kk^n$ if and only
if, by Lagrange interpolation, $\pi_r(P) \neq \pi_s(P) \neq 0$
whenever $r \neq s$. Then (1) follows.
We then consider $\pi_{i,j} : A \otimes_{\kk} A \onto \kk[X,Y]/(X- \zeta^i,
Y-\zeta^j) = \kk$ mapping $X \mapsto \zeta^i$, $Y \mapsto \zeta^j$
and $\pi = \bigoplus_{i,j} : A \otimes_{\kk} A \simeq \kk^{n^2}$.
Once again by Lagrange interpolation we get (2).
Then (3) follows from the identification $P = \la_0 + \la_1 X
+ \dots + \la^{n-1} X \mapsto \ula$ of $A$ with $\kk^n$,
for which $P(\zeta^i) = <\ula | v_i>$.

%Let now $v_i = 1 + \zeta^iX + (\zeta^i)^2 X^2 + \dots + (\zeta^i)^{n-1}
%X^{n-1}$ for $i \in [0,n-1]$, and introduce
%bilinear form $< X^i | X^j > = \delta_{ij}$ on $A$.
%Then $P(\zeta^i) = <P | \zeta^i>$, and the $v_i$ are linearly
%independant. Then $P(\zeta^i) - P(\zeta^j) - P(\zeta^k) + P(\zeta^l)
%= < P| v_i  - v_j + v_k - v_l>$
\end{proof}

We will also need the following lemma

\begin{lemma} \label{critsltens} Let $\g \subset \sl(V)$ be a complex semisimple Lie algebra
acting irreducibly on $V$, and $\rho : \g \to \sl(V)$
the defining representation such that $\rho^{\otimes 2}$
admits at most 3 irreducible components. Then $\g$ is a simple Lie algebra.
Moreover, if $\rho^{\otimes 2}$ admits two
irreducible components, then $\g = \sl(V)$ ; if
$\rho^{\otimes 2}$ admits three irreducible components
and $\rho^* \simeq \rho$, then $\g$ preserves some
nondegenerate bilinear form, and $\g = \osp(V_{\rho})$ w.r.t. this form.
\end{lemma}
\begin{proof}
If $\g = \g_1 \times \g_2$ with the $\g_i$ nontrivial
semisimple Lie algebras, then $V = V_1 \otimes V_2$ with
$V_i$ an irreducible representation of $\g_i$. But then
$\rho^{\otimes 2} = V_1^{\otimes 2} + V_1 \otimes V_2 + V_2 \otimes V_1
+ V_2^{\otimes 2}$ would admit at least 4 irreducible components. It
follows that $\g$ is simple. The conclusion follows from \cite{expo} prop. 1.
\end{proof}

\subsection{Generic infinitesimal Hecke algebras}

For $c \in \mathcal{C}_+$ we 
%let $v_i^c \in \kk^{e_C}$
choose a primitive $e_c$-root $\zeta_c$ of 1, and we
denote $v_i^c \in \kk^{e_c-1}$ the vector $\zeta_c^i, (\zeta_c^i)^2,\dots,
(\zeta_c^i)^{e_c-1}$. Identifying $\AAA(W)$ with
$\kk^{\sum_{c \in \mathcal{C}_+}
(e_c-1)}$ we denote $< \cdot | \cdot >$ the natural scalar product on $\AAA(W)$. 
We introduce in $\mathbbm{A}(W)$ the following hyperplane arrangements
$$
\begin{array}{l}
\mathcal{L}_1   = \{ \Ker <v_r^c - v_s^c| \cdot > \ | \ r \neq s , c \in \mathcal{C}_+ \} \\
\mathcal{L}_2   = \{ \Ker <v_i^c + v_j^c - v_k^c - v_l^c | \cdot > \ | \ \zeta_c^i + \zeta_c^j \neq \zeta_c^k + \zeta_c^l, c \in \mathcal{C}_+ \} \\
\mathcal{L}_3   = \{ \Ker <v_i^c - v_j^c - v_k^c + v_l^c | \cdot > \ | 
\ (i,j) \neq (k,l) \mbox{ and } (i\neq j \mbox{ or } k \neq l)
, c \in \mathcal{C}_+ \} \\
\end{array}
$$
Note that these three sets do not depend on the choice
of the primitive $e_c$-roots $\zeta_c$. Also note that $\mathcal{L}_1 \subset \mathcal{L}_2
\cap \mathcal{L}_3$.
For $\rho \in \Irr(W)$, we let $\rho_{\HH(\bula)}, \rho_{\HH(\bula)'}$
the representations of $\HH(\bula), \HH(\bula)'$ respectively,
that are induced by $\rho$. We let $\AAA^{\times}(W) = \AAA(W)
\setminus \bigcup (\mathcal{L}_1 \cup \mathcal{L}_2 \cup \mathcal{L}_3)
=
\setminus \bigcup (\mathcal{L}_2 \cup \mathcal{L}_3)
$.

Let $c$ be a conjugacy class in $W$. We let $p : \kk W \onto Z(\kk W)$
denote the natural projection $p(g) = (1/|W|)\sum_{h \in g} h gh^{-1}$.

Let $c$ be a conjugacy class of reflections in $\SS$. We denote $\SS/W$ the
set of such classes.
For $c \in \SS_0$ we denote $T_c(\bula) = T_c =\sum_{s \in c} s \in Z(\kk W)$ and for
$c \in \SS_+$ we let
$$T_c(\bula) = \sum_{s \in c } \sum_{i=1}^{e_c-1} \la^c_i s^i=
\sum_{i=1}^{e_c-1} \la^c_i \sum_{s \in  c } s^i \in Z(\kk W).
$$
For $s \in \SS_0$ we let $s(\bula) = s$ and for $s \in \SS_+$
we let $s(\bula) = \sum_{i=1}^{e_c-1}\la_i s^i$.
Clearly $T_c(\bula) = p(s(\bula))|c|$ for any $s \in c \subset \SS$.

\begin{prop} For every value of $\bula$, the Lie algebra $\HH(\bula)$
is reductive.
\end{prop}
\begin{proof}
We consider the direct sum of all the irreducible representations of
$W$. This provides a faithful representation of $\HH(\bula)$. In order
to prove that $\HH(\bula)$ is reductive it is sufficient to prove
that this representation is semisimple which means that,
for every $\rho \in \Irr(W)$, the representation of the envelopping
Lie algebra $\mathsf{U} \HH(\bula)$ induced
by $\rho$ is semisimple. Up to extending scalars we can assume $\kk = \C$,
and then assume that $\rho(W)$ preserves some unitary form over
$\End(V_{\rho})$. Then the elements $x = \la_1^c s + \dots + \la_{e_c-1}^c s^{e_c-1}$
act by normal operator (i.e. endomorphisms commuting with their
adjoints), which means that $\rho(x)$ is a polynomial
of its adjoint. In particular, if $U \subset V_{\rho}$ is
stable under $\rho(\mathsf{U} \HH(\bula))$, then its orthogonal
is stable under each of the elements $\rho(x)$, hence under $\rho(\mathsf{U}\HH(\bula))$.
As a consequence $V_{\rho}$ is completely reducible and the conclusion
follows. 
\end{proof}

\begin{prop}
\begin{enumerate}
\item For each $\rho \in \Irr(W)$, $\rho_{\HH(\bula)}$
and $\rho_{\HH(\bula)'}$ are irreducible as soon as $\bula \not\in \bigcup \mathcal{L}_1$.
%In that case, $\HH(\bula)$ is reductive and $\HH(\bula)'$ is semisimple.
\item If $\bula \not\in \bigcup \mathcal{L}_1$, the
center of $\HH(\bula)$ has dimension $|\SS/W|$, is spanned by
the $T_c(\bula)$ for $c \in \SS/W$ and $\HH(\bula)'$ is generated
by the $s(\bula) - T_c(\bula)/|c|$.

%\footnote{a demontrer !}
\item Let $\rho \in \mathrm{QRef}$ or $\rho \not\in
\LRef$, and $\bula \not\in \bigcup (\mathcal{L}_1 \cup \mathcal{L}_2)$.
Then $\rho(\HH(\bula)') = \rho(\HH')$.
\item For $\bula \in \AAA(W)^{\times}$, $\HH(\bula)' = \HH'$.
\end{enumerate}
\end{prop}
\begin{proof}
(1) For each $s \in \SS_+$ with reflection hyperplane $H \in c \in \mathcal{C}_+$,
$\rho(s)$ is a polynomial in $\sum \la^c_i s^i$ by the lemma
if $\bula \not\in \bigcup \mathcal{L}_1$. This implies that the
algebra generated by $\rho(\HH(\bula))$ and $\rho(\HH_s)$ are the
same, hence $\rho_{\HH(\bula)}$ is irreducible.
and $\rho(\HH(\bula))$ is reductive. This implies that
$\rho(\HH(\bula)') = \rho(\HH(\bula))'$ is semisimple and that
$\rho_{\HH(\bula)'}$
is irreducible.
%Then $\HH(\bula)$ admits a faithful semisimple
%representation, hence is reductive and $\HH(\bula)'$ is semisimple.
(2) Since, under this conditition, $\HH(\bula)$ generates
$\kk W$ as a unital algebra, we have $Z(\HH(\bula)) \subset Z(\kk W)$.
Since $\HH(\bula)$ is reductive we have $\HH(\bula) = Z(\HH(\bula))
\oplus \HH(\bula)'$ hence $p(\HH(\bula)) = p(Z(\HH(\bula)))
= Z(\HH(\bula))$.
Let $E \subset \kk W$ be the subspace spanned by the generators of
$\HH(\bula)$ and, for a conjugacy class $c$, denote $\delta_c$
the linear form on $\kk W$ defined by $\delta_c(g) = 1$
if $g \in c$, $\delta_c(g) = 0$ otherwise. We have
$(\kk W)' = \cap_c \Ker \delta_c = \Ker p$. Since $\HH(\bula)$
is generated by $E$ we have $\HH(\bula) = E +
\HH(\bula)'$ hence $p(\HH(\bula)) = p(E)$, hence $Z(\HH(\bula))$
is spanned by the $T_c(\bula)$ for $c \subset \SS$. Finally,
the $s(\bula) - p(s(\bula))$ generate a Lie algebra
containing $\HH(\bula)'$, as $p(s(\bula)) \in Z(\HH(\bula))$
and contained in $\HH(\bula) \cap \Ker p  = \HH(\bula)'$,
which concludes the proof of (2).
We now prove (3). For such a $\rho$, we let $\tilde{\rho}$ denote the
representation $\rho_{\HH_s'} \otimes \rho_{\HH_s'}$
extended to the envelopping algebra $\mathsf{U} \HH_s'$.
First assume that $\rho(\HH') = \rho(\HH_s') = \sl(V_{\rho})$. Then $\tilde{\rho}
(\mathsf{U} \HH_s') \simeq \End(S^2V_{\rho}) \oplus \End(\Lambda^2 V_{\rho})$.
Since $\bula \not\in \bigcup \mathcal{L}_2$ we get by the lemma that,
for each $s \in \SS_+$ with reflection hyperplane
$H \in c \in \mathcal{C}_+$,
$\tilde{\rho(s)}$ is a polynomial in $\tilde{\rho}(\sum \la^c_i s^i)$.
Then $\tilde{\rho}(\mathsf{U} \HH(\bula)') = \tilde{\rho}(\mathsf{U}
\HH_s')
= \End(S^2V_{\rho}) \oplus \End(\Lambda^2 V_{\rho})$.
In particular, $\rho_{\HH(\bula)'} \otimes \rho_{\HH(\bula)'}$
admits two irreducible components,
hence $\rho(\HH(\bula)') = \sl(V_{\rho})
= \rho(\HH')$ by lemma \ref{critsltens}. Now assume that $\rho(\HH') = \rho(\HH_s') = \osp(V_{\rho})$.
By the same argument, $\rho_{\HH(\bula)'} \otimes \rho_{\HH(\bula)'}$
admits three irreducible components
and $\rho_{\HH'} \simeq
(\rho_{\HH'})^* \Rightarrow 
\rho_{\HH_s'} \simeq
(\rho_{\HH_s'})^*$,
hence $\rho(\HH(\bula)') = \osp'(V_{\rho})$
for some nondegenerate bilinear form by lemma \ref{critsltens}. On the other hand,
$\rho(\HH(\bula)') \subset \rho(\HH') = \osp(V_{\rho})$,
hence $\osp'(V_{\rho}) \subset \osp(V_{\rho})$ which implies
$\osp'(V_{\rho}) = \osp(V_{\rho})$ and proves (3).
%We clearly have $\HH(\bula)' \subset \HH'$.
Let
$s \in \SS$  and $\rho^1,\rho^2 \in \Irr(W)$
with $\rho^1_{\HH(\bula)'} \simeq \rho^2_{\HH(\bula)'}$. Identifying
$V_{\rho^1} = V_{\rho^2} = V$, there
exists $Q \in \GL(V)$ and $\om_s \in \kk$ with $Q P_{\ula^c}(\rho^2(s)) Q^{-1}
= P_{\ula^c}(\rho^1(s)) + \om_s$ if $s \in \SS_+$ %has reflecting hyperplane
%in $c \in \mathcal{C}_+$
, and $Q \rho^2(s) Q^{-1} = \rho^1(s) + \om_s$
otherwise.
If $\rho^1(s) \not\in \kt$, for $s^2 = 1$ this implies
$\om_s = 0$ as in the proof of prop. \ref{propHHisom} ; otherwise,
$\rho^1(s)$ admits at least two eigenvalues $\zeta_c^i \neq \zeta_c^j$.
Then $P(\zeta_c^i) - \om_s = P(\zeta_c^k)$ and
$P(\zeta_c^j) - \om_s = P(\zeta_c^l)$ for some $k,l$
hence
$P(\zeta_c^i) - P(\zeta_c^k) - P(\zeta_c^j) + P(\zeta_c^l) = 0$.
If $\om_s \neq 0$ we have $i \neq k$ and $j \neq l$. Since $i \neq j$
we get a contradiction for $\bula \in \AAA(W)^{\times}$, thus $\om_s = 0$.
We thus get $\rho^2 = \rho^1 \otimes \eta$ for some $\eta \in \XX(\rho^1)$
and $\rho^1_{\HH'} \simeq \rho^2_{\HH'}$.
We now assume $(\rho^1_{\HH(\bula)'})^* \simeq \rho^2_{\HH(\bula)'}$.
For $s \in \SS$ with reflection hyperplane in $c \in \mathcal{C}_+$,
the equation $Q P_{\ula^c}(\rho^2(s))Q^{-1} = - ^t P(\rho^1(s)) + \om_s$
with $\rho^1(s) \not\in \kt$ with eigenvalues $\zeta_c^i \neq \zeta_c^j$,
implies that there exists $\zeta_c^k \neq \zeta_c^i$,
$\zeta_c^l \neq \zeta_c^i$ with $-P(\zeta_c^i)+ \om_s = P(\zeta_c^k)$,
$-P(\zeta_c^j)+ \om_s = P(\zeta_c^l)$ whence $P(\zeta_c^i) - P(\zeta_c^j)
+P(\zeta_c^k)- P(\zeta_c^l) = 0$. Since $\zeta_c^i \neq \zeta_c^j$
and $\zeta_c^k \neq \zeta_c^l$, we get from
$\bula \in \AAA(W)^{\times}$ that $(\zeta_c^i,\zeta_c^j) = (\zeta_c^k,\zeta_c^l)$.
But then $\om_s = 2 P(\zeta_c^i) = 2 P(\zeta_c^j)$ which implies
$\zeta_c^i = \zeta_c^j$, a contradiction. We conclude
as above that $\rho^2_{\HH'} \simeq (\rho^1_{\HH})^*$, and lemma
\ref{lemtheo} implies that $\HH(\bula)' = \HH'$, as $\HH' = \mathcal{M}$
by theorem \ref{theostruct}.
\end{proof}

Note that, if $\bula , \bumu \in \AAA(W)$ are such that, for all $c \in \SS_+/W$,
we have $\umu^c = u_c \ula^c$ for some $u_c \in \kt$, clearly $\HH(\bula) = \HH(\bumu)$.
Now denote $\alpha_c  \in \Hom(W,\kt)$ defined by $\alpha_c(s) = \zeta_c$
for $s \in c \subset \SS$, and $\alpha_c(s) = 1$ for $s \in \SS \setminus c$.
Any tuple $\underline{n} = (n_c)_{c \in \SS/W}$ of integers with $0 \leq n_c < e_c -1$
then defines a character $\alpha_{\underline{n}} = \prod_{c \in \SS/W} \alpha_c^{n_c}$
and we can define $\Phi_{\underline{n}} \in \Aut(\kk W)$ by $\Phi_{\underline{n}}(g) = 
\alpha_{\underline{n}}(g) g$ for $g \in W$. For any $\rho \in \Irr(W)$
we have $\rho \circ \Phi_{\underline{n}} \simeq \alpha_{\underline{n}} \otimes \rho$ and,
for any $\bula \in \AAA(W)$ we have $\Phi_{\underline{n}}(\HH(\bula)) = \HH(\bumu)$ with
$\mu^c_i = \zeta_c^{i n_c} \la^c_i$ for each $c,i$. The hyperplane arrangements $\mathcal{L}_i$
are clearly invariant under these operations.

\section{Applications to Zariski closures}

We recall that $W < \GL(V)$ is a finite group generated
by a set $\RR$ of (pseudo-)reflections that defines an hyperplane
complement $\mathcal{A} = \{ \Ker(s-1)\ | \ s \in \RR \}
= \{ \Ker(s-1) \ | \ s \in \SS\}$. The (generalized) pure braid
group and  braid group associated to them are $P = \pi_1(X)$
and $B = \pi_1(X/W)$, where $X$ is the hyperplane complement
$V \setminus \bigcup \mathcal{A}$ and a base point in $X$
(hence in $X/W$) is chosen once and for all.

We refer to \cite{BMR} to the basic properties of these groups,
only updating some terminology and recalling a few basic facts.

The composition of paths $(\alpha,\beta) \mapsto \alpha \beta$
is given by following first $\beta$ and then $\alpha$. With
this convention, the Galois covering $X \to X/W$
defines a natural morphism $B \onto W$ with kernel $P$.
We call the generators-of-the-monodromy in $B$ which were
associated to elements of $\RR$ in \cite{BMR} \emph{braided reflections}.
Recall from \cite{BMR} that they generate $B$.

An important structure associated to the hyperplane complement
$X$ is its holonomy Lie algebra $\mathcal{T}$ (see e.g. \cite{KOHNOKOSZUL}).
It is generated by one element $t_H$ for each hyperplane $H \in \mathcal{A}$,
and there is a natural action of $W$ by automorphisms of $\mathcal{T}$,
given by $w.t_H = t_{w(H)}$. To each $H \in \mathcal{A}$ is
associated a (well-defined) logarithmic 1-form $\om_H = \dd \alpha_H/\alpha_H$,
for an arbitrary 1-form $\alpha_H \in V^*$ with kernel $H$.

To any representation $\varphi : \mathcal{T} \to \gl_N(\C)$ is
associated a (family of) integrable 1-form(s)
$$
\om = h \sum_{H \in \mathcal{A}} \varphi(t_H) \om_H
$$
over $X$ with values in $\gl_N(\C)$, depending on some complex
parameter $h$. By monodromy this family of 1-forms defines a
family of representations
of $P = \pi_1(X)$,
that can be considered as a representation $P \to \GL_N(A)$
with $A = \C[[h]] \subset K = \C((h))$.

If $\C^N = V_{\rho}$ for some linear representation $\rho$
of $W$ such that $\varphi : \mathcal{T} \to \gl(V_{\rho})$
is equivariant w.r.t. the natural actions of $W$, then the
representation of $P$ naturally extends to a representation
$B \to \GL(V_{\rho} \otimes A)$. A change of basepoint
gives rise to a representation which is isomorphic to
the former one.

In \cite{BMR} is introduced a representation
of $\mathcal{T}$, depending on a collection $\tau_{c,j}$ of
complex numbers for $c \in \mathcal{A}/W= \SS/W$,
where $\SS$ here denotes the set of distinguished
reflections, and $0 \leq j \leq e_c-1$. For $H \in c$
and $s_H \in \SS$ with $\Ker(s-1) = H$, it is given by
$$
\varphi(t_H) = \sum_{j=0}^{j=e_c-1} \tau_{H,j} \frac{1}{e_c}
\sum_{k=0}^{k = e_c-1} \zeta_c^{-jk} \rho(s_H^k) = 
\sum_{k=0}^{e_c-1} \frac{1}{e_c}  \left( \sum_{j=0}^{j=e_c-1} (\zeta_c^{-k})^j \tau_{H,j} \right) \rho(s_H)^k
$$
with $\zeta_c = \det(s_H) = \exp(2 \ii \pi/e_c)$. The representation
of $B$ associated to that satisfies
$$
\prod_{j=0}^{e_c - 1} (\sigma - q_{H,j} \zeta_c^j ) = 0
$$
for any braided reflection associated to $s_H$ and $q_{H,j} = \exp(-h \tau_{H,j}/e_c)$,
and deforms $\rho$ into a representation of the Hecke algebra
of $W$, as defined in \cite{BMR}, with these parameters.

We note that the specialization of interest in the Brou\'e-Malle-Michel
`Spetses' program (see \cite{BMM}) is for $q_{H,0} = q$ and $q_{H,j} = 1$ for $j \neq 0$.
This specialization corresponds to the choice of parameters
$\tau_{H,j} = 0$ for $j \neq 0$.

We let $R : P \to \GL(V_{\rho} \otimes K)$ denote the
representation of $P$ associated to $\rho$. We refer to \cite{KRAMCRG}
for a proof of the next proposition.

% and $\overline{R(P)}$
%its Zariski closure, the Lie algebra of $\overline{R(P)}$ contains
%$\varphi(\mathcal{T}) \otimes K$, by .

\begin{prop} \label{propzarchev} For any representation $\rho$ of $W$, the Lie algebra of
the Zariski closure of $R(P)$ contains $\varphi(\mathcal{T}) \otimes K$.
If $\varphi(\mathcal{T})$ preserves some bilinear form over
$V_{\rho}$, then $R(P)$ preserves the induced bilinear form
over $V_{\rho} \otimes K$.
\end{prop}

The group $P$ is generated by the $\sigma^{e_c}$ for $\sigma$
running among the braided reflections, with $c$  denoting
the class of the hyperplane attached to $\sigma$.
We have
$$\det R(\sigma^{e_c}) = \exp(h e_c \mathrm{tr} \varphi(t_H)) =
\exp \left( h  \sum_{k=0}^{e_c-1}   \left( \sum_{j=0}^{j=e_c-1} (\zeta_c^{-k})^j \tau_{H,j} \right) \mathrm{tr} \rho(s_H)^k \right)
$$
We let $\la^c_k =  \sum_{j=0}^{j=e_c-1} (\zeta_c^{-k})^j \tau_{c,j}$.
Note that this defines a bijection between the $(\la^c_k)$ and
the $(\tau_{c,k})$ by invertibility of the Vandermonde determinant.
With this notation, $\varphi(t_H) = (1/e_c) \sum_{k=0}^{e_c-1} \la^c_k
\rho(s_H)^k$.

Also note that $\la^c_0 = \sum_{j=0}^{e_c-1} \tau_{c,j}$. Up to
tensoring $R$ by the 1-dimensional representation that
maps a braided reflexion around $H \in \mathcal{A}$ of class $c$
to $\exp \left( h (\sum_{j=0}^{e_c-1} \tau_{c,j})/e_c\right)$, we can
assume $\la^c_0 = 0$, which we do from now on.

With this convention, we have $\varphi(\mathcal{T}) = \rho(\HH(\bula))$,
hence $\varphi(\mathcal{T})$ is a reductive Lie algebra. From
proposition \ref{propzarchev} and theorem \ref{theostruct} one
readily gets the following.

%With this notation, $\varphi(t_H) = (1/e_c) \sum_{k=0}^{e_c-1} \la^c_k
%\rho(s_H)^k$, and
\begin{theo}\label{theozarheck}  %We have $\varphi(\mathcal{T}) = \rho(\HH(\bula))$.
If $\bula \in
\AAA(W)^{\times}$ then $\varphi(\mathcal{T})$ is reductive,
$\varphi(\mathcal{T})' = \rho(\HH')$ and $R(P)$ has connected Zariski
closure, with Lie algebra $\rho(\HH(\bula))$.
\end{theo}

\section{Unitarizability questions}

The determination of the Zariski closure is specially useful
when the representations involved are unitarizable. The monodromy
construction as described above provides a morphism $B \to W \ltimes \exp
\widehat{\mathcal{T}}$, where $\mathcal{T}$ is completed with respect to
the graduation $\deg t_H = 1$.

Recall from \cite{KOHNOKOSZUL} that $\mathcal{T}$ can be defined over an arbitrary field $\kk$
of characteristic 0. We state the following conjecture :

\begin{conj} \label{conj1} For an arbitrary complex (pseudo-)reflection group $W$
and field $\kk$ of characteristic 0,
there exists morphisms $\Phi : B \to W \ltimes \exp \mathcal{T}$,
with $\mathcal{T}$ defined over $\kk$, such that
$\Phi(\sigma)$ is conjugated to $s_H \exp t_H$ by some element
in $\exp \widehat{\mathcal{T}}$ for every braided reflection $\sigma$
associated to $s_H$ and $H \in \mathcal{A}$.
\end{conj}

Evidences for this conjecture include the fact that it holds
for $W = \mathfrak{S}_n$ by using a rational associator as defined
by Drinfeld (see \cite{DRIN}), for $W$ of type $G(d,1,n)$ by using the analogous
gadgets by Enriquez (\cite{ENRIQUEZ}; see also the appendix of \cite{DIEDRAUX}),
and for $W$ of type $G(e,e,2)$ (see \cite{DIEDRAUX}).
Another evidence is that all the varieties involved are defined
over $\Q$, as proved in \cite{MARINMICHEL}.

Recall that it is conjectured in \cite{BMR} that the Hecke algebra
of $W$ is always a flat deformation of the group algebra of $W$.
In the sequel, we call this the BMR-conjecture. It is known to
hold for all but possibly a finite number of exceptional cases.

Notice that a necessary condition for the representation $R$ defined in
a previous section to be unitary is that the $q_{H,j} \zeta_c^j$ have modulus 1.
When $h$ is specialized to a purely imaginary number, this happens exactly
when the $\tau_{H,j}$ are real numbers, for $0 \leq j \leq e_c - 1$.
%Recall from the previous section that we can assume $\la^c_0 = \sum_{j=0}^{e_c-1} \tau_{c,j} = 0$.
Since, for $e_c - 1 \geq k \geq 1$,
$$
\la_k^c - \overline{\la^c_{e_c - k}} = \sum_{j=1}^{e_c - 1} (\zeta_c^{-k})^j
(\tau_{c,j} - \overline{\tau_{c,j}}),
$$
this imposes that $\la^c_k$ is the complex conjugate of
$\la^c_{e_c-k}$.

\begin{theo} \label{theounit} If conjecture \ref{conj1} and the BMR-conjecture hold for $W$ when $\kk = \R$ and, for all $c \in \mathcal{C}_+$
and $0 \leq j \leq e_c-1$, $\tau_{c,j} \in \R$,
%and $1 \leq k \leq e_c-1$, $\la^c_k$ is the complex conjugate of
%$\la^c_{e_c-k}$, 
then $R$ is unitarizable for $h \in \ii \R$ small enough.
\end{theo}

%Notice that this happens, with the notations of the
%previous sections, exactly when $\tau_{c,j} \in \R$, as

\begin{proof}
Up to tensoring the representation by a 1-dimensional unitary
one, we can assume that $\la^c_0 = 0$ for every $c \in \mathcal{C}^+$.
The condition then exactly means that, for all $1 \leq k \leq e_c-1$,
$\la^c_k$ is the complex conjugate of
$\la^c_{e_c-k}$.

We introduce the following automorphisms of $K = \C((h))$. Let
$\eps \in \Aut(K)$ being defined through $f(h) \mapsto f(-h)$, and
$\bar{\eps}$ be the its composed by the complex conjugation
$\sum_{r \geq r_0} a_r h^r \mapsto \sum_{r \geq r_0} \overline{a_r} (-h)^r$.
We can assume $\rho(W) \subset U_N$. We let
$U_N^{\bar{\eps}}(K) = \{ x \in \GL_N(K) \ | \ ^t\bar{\eps}(x) = x^{-1} \}$.
Since $\varphi(t_H) = (1/e_c)\sum_{k=0}^{e_c - 1} \la_k^c \rho(s_H)^k$
and since the $\rho(s_H)^k$ are unitary, our assumption implies that
$\varphi(t_H)$ is selfadjoint, hence $\exp h \varphi(t_H) \in U_N^{\bar{\eps}}(K)$.
Let $\sigma$ be a braided reflection
corresponding to $s_H$, and $\tilde{\varphi} : \widehat{\mathcal{T}} \to
\gl_N(K)$ be defined through $t_H \mapsto h \varphi(t_H)$.
Now consider the automorphism $A \in \Aut(\gl_N(K))$ given by $x \mapsto - ^t x$.
We have $A \circ \tilde{\varphi}(t_H) = - ^t h \varphi(t_H)
= - h \overline{\varphi(t_H)} = \bar{\eps} \circ \tilde{\varphi}(t_H)$.
It follows that $A \circ \tilde{\varphi} = \bar{\eps} \circ \tilde{\varphi}$
and in particular 
$ - ^t \tilde{\varphi}(\psi) =  A \circ \tilde{\varphi}(\psi) = \bar{\eps} \circ \tilde{\varphi}(\psi)
= \bar{\eps} (\tilde{\varphi}(\psi))$. Taking exponentials,
we get
$$
\left( ^t\exp(  \tilde{\varphi}(\psi)) \right)^{-1} = \exp ( - ^t \tilde{\varphi}(\psi))
= \exp ( \bar{\eps} (\tilde{\varphi}(\psi)) ) = \bar{\eps} 
\left( \exp((\tilde{\varphi}(\psi)) \right)
$$
whence $\exp \tilde{\varphi}(\psi) \in U_N^{\bar{\eps}}(K)$.
Finally note that $\rho(s_H) \in U_N \subset U_N^{\bar{\eps}}(K)$.
Now $(\rho, \tilde{\varphi})$ provide a morphism $W \ltimes \exp
\widehat{\mathcal{T}} \to \GL_N(K)$. Combining it with
$\Phi$ we thus get a representation $R' : B \to U_N^{\bar{\eps}}(K)$,
that factorizes through the Hecke algebra. Under the BMR-conjecture,
we have $R \simeq R'$, since both $R$ and $R'$ specialize to
$\rho$ through $h = 0$.

Let $L$ be a subfield of $K$ containing $\R$, the entries of the $R'(b)$
for $b \in B$ as well as $\mathrm{i} \in \C$ and $h$. Since
$B$ is finitely generated we can take for $L$ a finitely generated
extension of $\R(h)$. Up to considering $L + \eps(L)$
we can assume $\eps(L) = L$. Since $\ii \in L$ we have
$L = L_0 \oplus \ii L_0$ with $L_0 \in \R((h))$ and $\eps(L_0) = L_0$.
By \cite{BMW} proposition 3.1 we know that $L_0$ is isomorphic to some finitely generated
extension $L_0^*$ of $\R(h)$ inside the field of convergent Laurent
series $\R(\{ h \})$, through a $\eps$-equivariant isomorphism.
It extends to $\bar{\eps}$-equivariant isomorphism
between $L = L_0 + \ii L_0$ and $L^* = L_0^* + \ii L_0^*$. 
Since $R'(B) \subset U_N^{\bar{\eps}}(L) \simeq U_N^{\bar{\eps}}(L^*)$
we get a new representation $R'' : B \to U_N^{\bar{\eps}}(\C(\{ h \}))$.
The isomorphism $L_0 \to L_0^*$ can be chosen such that the entries
of the $R(b)$, for $b$ running inside a finite set of generators for $B$,
are unchanged modulo $h$. By the same argument as above, under the
BMR-conjecture, $R''$ is isomorphic to $R'$ hence to $R$ over $\C(( h ))$.
Since $R''$ and $R$ are both defined over $\C(\{ h \})$, they
are conjugated over this field. Since they coincide modulo
$h$ they are both defined over the ring of convergent power series,
so their specialisations to small $h$ are isomorphic over $\C$.
Now, for $h \in \ii \R$ small enough, the specialization of $R''$ is unitary,
and this concludes the proof.
\end{proof}

\begin{cor} If $W$ is a complex reflection group of type $G(de,e,r)$,
then $R$ is unitarizable under the above conditions on the parameters.
\end{cor}
\begin{proof} The BMR-conjecture is known to hold for
$W$ of these types. Since conjecture \ref{conj1} holds for $W$
of type $G(de,1,r)$ by \cite{ENRIQUEZ}, we get unitarity
for these groups. Now the Hecke algebra of type $G(de,e,r)$
is a subalgebra of the one of type $G(de,1,r)$, with images of braided
reflexions mapped to images of braided reflections (see \cite{RR}); since
any irreducible representation of the Hecke algebra of type $G(de,e,r)$
appears in the restriction of one of type $G(de,1,r)$ (see \cite{RR}),
this concludes the proof of the corollary.
\end{proof}

We recall from \cite{IH2} that the unitarizability is known for
Coxeter groups, and was also proved for the reflection
representation by geometric methods in \cite{CHL}.

In the unitary cases, and for transcendant values of the
parameters, the Lie algebra of the topological closure
is then given by a compact form of the Lie algebra of the
Zariski closure. Such a compact form is described
by the following generalization of \cite{IH2} prop. 2.27.

\begin{prop} Assume that, for all $c \in \mathcal{C}_+$
and $1 \leq k \leq e_c-1$, $\la^c_k$ is the complex conjugate of
$\la^c_{e_c-k}$. Then the real Lie subalgebra $\mathcal{H}_c(\bula)$
of $\C W$ generated by the $\ii s, s \in \SS_0$ and the
$\ii (\sum \la^c_k s^k)$ for $s \in c \subset \SS_+$ is a compact real
Lie subalgebra of $\mathcal{H}(\bula)$. Moreover, $\HH_c(\bula)'$ is
a compact real form of $\HH(\bula)'$.
\end{prop}
\begin{proof}
We define on $\C W$ a sesquilinear form by $(w_1,w_2) = \delta_{w_1,w_2}$
for $w_1,w_2 \in W$. This form is clearly positive definite. By left action we have $\mathcal{H}(\bula) \subset
\C W \subset \End(\C W)$, and the generators of $\mathcal{H}_c(\bula)$
are easily checked to satisfy, under our assumption, the
equality $x^* = -x$, where $x^*$
denotes the adjoint of $x \in \End(\C W)$ with respect to our form.
This proves that $\mathcal{H}_c(\bula)$ is a real Lie subalgebra of
the compact Lie algebra $\mathfrak{u}(\C W)$, and therefore is
a compact Lie algebra. Now $\HH_c(\bula)' + \ii \HH_c(\bula)$
is a complex Lie subalgebra of $\HH(\bula)'$, which equals
$\HH(\bula')$ because every iterated bracket $[x_1,[x_2,\dots,[x_{r-1},x_r]\dots]$ in the
defining generators of $\HH(\bula)$ can be written as
$(-\ii)^r [y_1,[y_2,\dots,[y_{r-1},y_r]\dots]$ where the $y_k = \ii x_k$ are the defining generators of $\HH_c(\bula)$.
Since $\HH_c(\bula)'$ is a compact Lie subalgebra of the semisimple
Lie algebra $\HH(\bula)'$, its real dimension is at most
$\dim_{\C} \HH(\bula)' = (\dim_{\R} \HH(\bula)')/2$
(otherwise it would define a compact subgroup of the semisimple
group $\exp \HH(\bula)'$,
of real dimension larger than its maximal compact subgroups).
It follows that $\HH_c(\bula)' \cap \ii \HH_c(\bula)' = \{ 0 \}$
hence $\HH_c(\bula)'$ is a compact real form of $\HH(\bula)'$.
%Finally, as $\mathcal{H}
%(\bula) \subset \C W$, we have $\mathcal{H}_c(\bula) $ \dots
\end{proof}

\section{Special situations}

\subsection{The special Hecke algebra}

Let $W$ an irreducible reflection group, and $\rho^1,\rho^2 \in \Irr(W)$.
We assume that $\rho^1_{\HH_s'} \simeq \rho^2_{\HH_s'}$.
We can assume that $\rho^2(s) = \rho^1(s) + \om_s$ for
all $s \in \SS$. Note that $\om_s \in \kk$ only depends on
the conjugacy class of $\SS$. If $\mathcal{X} = \Sp \rho^1(s)$
we have $\mathcal{X} + \om_s = \Sp \rho^2(s) \subset \mu_n$,
where $n$ is the order of $s$. By lemma \ref{lemcercle} we have $\om_s \neq 0
\Rightarrow |\Sp(\rho^2(s))| \leq 2$. Up to tensoring
$\rho^1$ by some character in $\XX(\rho^1)$, we can assume that,
for all $s \in \SS$, either $\rho^2(s) = \rho^1(s)$, or
$|\Sp(\rho^2(s))| = 2$ when $\om_s \neq 0$.
When $\Sp(\rho^2(s)) = \{ \alpha,\beta \}$
with $\alpha \neq \beta$, then $\om_s = \alpha + \beta$
and $\Sp \rho^1(s) = \{ - \alpha, - \beta \}$. This is only possible
for $n$ an even integer. Also note that, if $n = 2$,
then necessarily $\om_s = 0$. We can thus assume that $n >2$
and $n$ is even.

Since $\rho^1(s)$ is semisimple, $\rho^1(s^{-1}) = (- \alpha^{-1} \beta^{-1})
(\rho^1(s) + \om_s)$, hence there exists $\chi \in \Hom(W,\kt)$
defined by $\chi(s) = 1$ if $\om_s = 0$ and $\chi(s) = (-1/\alpha\beta)$
when $\om_s \neq 0$ and $\Sp \rho^1(s) = \{ \alpha,\beta \}$
with $\alpha \neq \beta$, such that $\rho^3 = \rho^2 \otimes \chi$ is an
irreducible representation of $W$ that satisfies
$\rho^3(s) = \rho^1(s^{-1})$ for all $s \in \SS$ with $\om_s \neq 0$
and $\rho^3(s) = \rho^1(s)$ when $\om_s = 0$.

Using the Shephard-Todd classification we can check on
the representations of the exceptional groups that these situations
do not occur for $W$ an exceptional group. The first remark
is that the orders of the pseudo-reflections for such a group
are at most 5, so the only cases to consider are when
$W$ admits a pseudo-reflection of order 4. There are four such groups,
namely $G_8,G_9,G_{10}$ and $G_{11}$. In these four cases,
there are six pseudo-reflections of order 4, and only one class of
such pseudo-reflections in $\SS$. We then restrict to the
$\rho^1 \in \Irr(W)$ such that the formulas $s \mapsto \rho^1(s^{-1})$ if
$s \in \SS$ has order 4 and $s \mapsto \rho^1(s)$ otherwise define
a representation of $W$. All such representations have dimensions 1 or 2.
Then the representation $\rho^2$ is given by
$\rho^2(s) = - (\det \rho^1(s)) \rho^1(s^{-1})$ when $s\in \SS$ has order 4
and $\rho^2(s) = \rho^1(s)$ otherwise, and it is readily checked
that $\rho^2$ has the same character as $\rho^1$ in all cases.
This proves the following.

\begin{prop} \label{propduauxexcept} Let $W$ be an exceptional irreducible
reflection group. For $\rho^1,\rho^2 \in \Irr(W)$, 
$\rho^1_{\HH_s'} \simeq \rho^2_{\HH_s'}$ if and only
if
$\rho^1_{\HH'} \simeq \rho^2_{\HH'}$,
that is $\rho^2 = \rho^1 \otimes \eta$ for some $\eta \in \XX(\rho^1)$.
\end{prop}

We now consider $W = G(de,e,r)$, for which we can assume with $d > 2$.
The elements in $\SS$ of order more than 2 have order $d$
and form a single conjugacy class $c \subset \SS$.
We can thus assume $d$ even, and let $\zeta = \zeta_c$.
The elements of $c$ are then the $t_i = \mathrm{diag}(1,\dots,1,\zeta,1,\dots,1)$
for $1 \leq i \leq r$. Is is easily checked (see e.g. \cite{MARINMICHEL})
that the formulas $t_i \mapsto t_i^{-1}$ define (uniquely)
an automorphism $\mathfrak{c}$ of $W$ that fixes $\mathfrak{S}_r \subset
W$.

It follows that $\rho^2_{\HH_s'} \simeq \rho^1_{\HH_s'}$
if and only if either $\rho^2_{\HH'} \simeq \rho^1_{\HH'}$
or $\rho^1(t_1)$ has two eigenvalues $\alpha \neq \beta$ and,
up to tensoring by some character in $\XX(\rho^1)$,
$\rho^2 = \chi \otimes \rho^1 \circ \mathfrak{c}$,
where $\chi \in \Hom(W,\kt)$ is defined by $\chi(t_i) = - \alpha \beta$,
$\chi(s) = 1$ for $s \not\in c$.

\begin{prop} Let $W = G(de,e,r)$ with $d > 2$ and
$\rho^1,\rho^2 \in \Irr(W)$. If $d$ is odd or
$\Sp(\rho^1(t_1)) \neq 2$ then $\rho^1_{\HH_s'} \simeq
\rho^2_{\HH_s'}$ iff $\rho^1_{\HH'} \simeq \rho^2_{\HH'}$.
\end{prop}

The following is then a consequence of lemma \ref{lemtheo}.

\begin{cor} For $G(de,e,r)$ with $d$ an odd integer,
or $W$ an exceptional group,
$\HH'_s = \HH'$.
\end{cor}

We consider the special case $e = 1$, and $\rho^1,\rho^2$
with $\rho^1_{\HH'} \not\simeq \rho^2_{\HH'}$
but $\rho^1_{\HH_s'} \simeq \rho^2_{\HH_s'}$. By the arguments
above $d>2$ is even, and $\rho^1(t)$ admits two eigenvalues.
This means that $\rho^1$ is labelled by
a multipartition $(a_0,\dots,a_{d-1})$ with two non-empty parts
$a_i,a_j$ with $i \neq j$, and $\Sp \rho^1(t_1) = \{ \zeta^{a_i},
\zeta^{a_j} \}$ with $\zeta = \exp(2 \ii \pi/d)$.
Then $\rho^3 = \rho^1 \circ \mathfrak{c}$ corresponds
to the multipartition $(b_0,b_1 , \dots, b_{d-1}) = (a_0,a_{d-1},a_{d-2},\dots,a_2,a_1)$,
and tensoring by $\chi$ with $\chi(t_1) = - a_i a_j = \zeta^k$ with $\zeta^k =
-\zeta^{i+j} = \zeta^{d/2 + i + j}$, e.g. $k = d/2 + i + j$,
leads to $\rho^2 = \rho^3 \otimes \chi$
labelled by $(b_{0-k},b_{1-k},\dots, b_{d-2-k},b_{d-1-k})$.
For instance, if $d=4,r=3$ and $\rho^1$ is labelled by $([2],[1],\emptyset,\emptyset)$,
we have $\XX(\rho^1) = \{ \un \}$
and the only possibility is for $\rho^2$ labelled by $(\emptyset,\emptyset,[1],[2])$.
More generally, if $\rho^1$ is the defining representation of $W$, labelled
by $([r-1],[1],\emptyset)$, then $\XX(\rho^1) = \{ \un \}$, but
$\rho^1_{\HH_s'} \simeq \rho^2_{\HH_s'}$ with $\rho^1 \not\simeq \rho^2$ as the formulas above guarantee, for $d> 2$
and $r \geq 3$, that the first part of (the multipartition labelling) $\rho^2$
will not be $[r-1]$.

For $e > 1$, this example also provides $\rho^1,\rho^2$ with $\rho^1_{\HH'} \not\simeq \rho^2_{\HH'}$
but $\rho^1_{\HH_s'} \simeq \rho^2_{\HH_s'}$. As a consequence,
we get that the above corollary is sharp.

\begin{prop} If $d > 2$ is even and $r \geq 3$, then $\HH_s' \not\simeq \HH'$.
\end{prop}

\subsection{Cubic Hecke algebras}

We assume that, for all $s \in \SS$, $s^3 = 1$, and that there
exists a single class $c$ of reflecting hyperplanes for $W$. For irreducible
reflection group of rank at least 2, this happens exactly for the exceptional
types $G_4,G_{25},G_{32}$. In that case we have two parameters $\la_1,
\la_2$, $\AAA(W) = \{ (\la_1,\la_2) \in \kk^2 \}$ and $\HH(\la_1,\la_2)$
is generated by the $\la_1 s + \la_2 s^2$ for $s \in \SS$.
Leting $j = \zeta_c = \exp(2 \ii \pi/3)$, we have $v_0 = (1,1)$, $v_1 = (j,j^2)$,
$v_2 = (j^2,j)$. It is easily checked that we have here $\mathcal{L}_3 \subset \mathcal{L}_2$
hence $\AAA(W)^{\times} = \AAA(W) \setminus \bigcup \mathcal{L}_2$. We have
$$
\begin{array}{lcl}
\mathcal{L}_1 &=& \{ \Ker < v_0 - v_1 | \cdot>, \Ker < v_0 - v_2 | \cdot>,
\Ker < v_1 - v_2 | \cdot> \} \\
\mathcal{L}_2 &=& \mathcal{L}_1 \sqcup \{ \Ker < v_0 + v_1 -2v_2| \cdot>,
\Ker < v_0 +v_2 - 2v_1 | \cdot>, \\
 & & \Ker < v_1 +v_2 -2 v_0| \cdot> \} \\
\end{array}
$$

It is clear that $\HH(0,0) = \{ 0 \}$, and we know $\HH(0,1) \simeq \HH(1,0) = \HH_s$,
which has already been studied (moreover $(1,0)$ and $(0,1)$ clearly
belong to $\AAA(W)^{\times}$), so we can assume $\la_1 \la_2 \neq 0$.
Then $\HH(\la_1,\la_2) = \la_1 \HH(1,\la_2/\la_1)$, so we can assume
$\la_1 = 1$ and $\la_2 = a \in \kt$ and let $\HH(a) = \HH(1,a)$.
A straightforward computation shows $(1,a) \in \bigcup \mathcal{L}_1$ iff
$a \in \mu_3$, and $(1,a) \in \bigcup \mathcal{L}_2$ iff $a \in \mu_3 \cup (- \mu_3)$.
In particular $(1,a) \in \AAA(W)^{\times}$ iff $a \not\in \mu_6$.

For $a^6 \neq 1$ (and $a \neq 0$), we have $\HH(\ula)' = \HH'$.
Letting $\Phi \in \Aut(\kk W)$ denote as above the automorphism $s \mapsto j s$
for $s \in \SS$, we have $\Phi(\HH(a)) = \HH(j,j^2a) = j\HH(1,ja) = \HH(1,ja) = \HH(ja)$,
so there are only two cases to consider, $a = 1$ for $a \in \mu_3$,
and $a = -1$ for $a \in - \mu_3$.

%We let $\mathcal{H}(a)$ denote the Lie algebra
%generated by the $s + a s^2$ for $s \in \SS$, that is $\bula = \ula = (1,a)$.
%The condition $\ula \not\in \mathcal{L}_1$ is $a^3 \neq 1$
%and the condition $\ula \not\in \mathcal{L}_2$ is $a^6 \neq 1$.

If $W = G_4$,
there are three representations in $\overline{\QRef} = \QRef = \LRef$
of dimension 2, three 1-dimensional characters and a 3-dimensional
one that we denote $\rho_3$. The 2-dimensional ones are the restrictions
to a parabolic subgroup of type $G_4$ of the representations
$U_{\alpha,\beta}$ of $G_{25}$. In order to avoid confusion, we denote
them $U_{\gamma}$ for $\{ \alpha,\beta,\gamma \} = \mu_3(\C)$. The 3-dimensional
one is the restriction $\bar{V}$ of $V \in \Irr(G_{25})$.
We assume that $\SS$ is given by the distinguished pseudo-reflections.
% and
%choose $\zeta = \zeta_c = j = \exp(2 \ii \pi/3)$.

We have $\HH' = \sl(U_1) \times \sl(U_j) \times \sl(U_{j^2}) \times \sl(\bar{V}) \simeq (\sl_2)^3 \times \sl_3$.
For $a = 1$, we get by computer that $\dim \HH(1) = 15$,
$\dim \HH(1)' = 14$, $\dim Z(\HH(1)) = 1$,
and that the image in each irreducible representation of $W$ of
$\HH(1)'$ is the same as $\HH'$, except for $U_{1}$, where the image
is 0. Since $\HH(1)'$ is semisimple this implies
$\HH(1)' = \sl(U_j) \times \sl(U_{j^2}) \times \sl(\bar{V})
\simeq \sl_2^2 \times \sl_3$.
% with $\{ a, b, c \} = \mu_3(\C)$.
Moreover, the center of $\HH(1)$, having dimension 1,
is also spanned by $T_{\SS}(1) = \sum_{s \in \SS} s +  s^2$
in this case.
Another argument
for this last fact, that will be used for $G_{25}$, is to notice that the
%\footnote{verifie que la meme chose se passe bien pour $G_{25}$ : si $\{a,b,c \} = \mu_3$,
%sur $U_{b^2,c^2}$ on a action de $s + a s^2$ par $-a^2 \Id_2$, et sur
%l'autre representation \og nulle \fg\ on a action par $-a^2 \Id_3$.}
%image of $\HH(\ula)$ itself is 0 in $U_a$ --- pas vrai !!!
image of $\HH(1)$ in $U_1$ is $\C$ (more precisely, $s +  s^2$ acts by $-1$).
It follows that $Z(\HH(1)) \subset Z(\HH) \subset Z(\C W)$.
Since $\HH(1)$ is reductive and $\HH(1)' \subset \Ker p$,
we thus recover $Z(\HH(1)) = p(\HH(1)) = \C T_{\SS}(1)$.

For $a =-1$, we get by computer $\dim \HH(-1)' = 6$,
and the image in each of the irreducible representations of $W$
of $\HH(-1)'$ is the same as $\HH'$, except for $\bar{V}$, for
which the image is of dimension 3. Moreover, since $a \not\in \mu_3$,
$\HH(-1)'$ is semisimple. We check that the image of $\HH(-1)'$
in $\sl(U_j) \times \sl(U_{j^2}) \times \sl(\bar{V})$ has dimension 3,
hence $\HH(-1)' \simeq \sl(U_1) \times \sl_2 \simeq (\sl_2)^2$, the restriction
of the three representations $U_j,U_{j^2},\bar{V}$ to $\HH(\ula)'$ factorizing
through the same ideal $\sl_2$. More previsely, we can identify
$\HH(-1)'$ with $\sl(A_1) \times \sl(A_2)$ with $A_1,A_2$ two vector
spaces with $\dim A_i = 2$, such that the representations of
$\HH(-1)'$ corresponding to $U_1,U_j,U_{j^2},\bar{V}$
are $A_1,A_2,A_2, S^2 A_2$, as $\sl_2$ admits only one irreducible
representation in each dimension.

%For $a \in \mu_6 \setminus \mu_3$, we get by computer $\dim \HH(\ula)' = 6$,
%and the image in each of the irreducible representations of $W$
%of $\HH(\ula)'$ is the same as $\HH'$, except for $\bar{V}$, for
%which the image is of dimension 3. Moreover, since $a \not\in \mu_3$,
%$\HH(\ula)'$ is semisimple. We let $\{a,b,c \} = \mu_3$
%and check that the image of $\HH(\ula)'$
%in $\sl(U_b) \times \sl(U_c) \times \sl(\bar{V})$ has dimension 3,
%hence $\HH(\ula') \simeq \sl(U_a) \times \sl_2 \simeq (\sl_2)^2$, the restriction
%of the three representations $U_b,U_c,\bar{V}$ to $\HH(\ula)'$ factorizing
%through the same ideal $\sl_2$.

We now consider $W = G_{25}$. For $a = 1$, we get $\rho(\HH(1)') = \rho(\HH')$
for each $\rho \in \Irr(W)$ except $U_{j,j^2}$ and $U'_{j^2,j}, U'_{j,j^2}$
(recall that $U'_{j,j^2} \approx U'_{j^2,j}$), in which case $\rho(\HH(1)') = \{ 0 \}$.
Moreover, we check that $\rho(\HH(1)) = \C$, and more precisely $\rho(s + s^2) = -1$
in these latter cases. %As before this proves that $\HH(1)$ is reductive.
We check by computer that, for any $\rho^1,\rho^2$ not in these
exceptional cases and $\rho^1 \not\approx \rho^2$, then $(\rho^1 \oplus \rho^2)(\HH(1)')$
has dimension  $\dim \rho^1(\HH(1)') + \dim \rho^2(\HH(1)')$, hence
$(\rho^1 \oplus \rho^2)(\HH(1)') = \rho^1(\HH(1)') \oplus \rho^2(\HH(1)')$.
This proves that the simple ideals determined by the non-exceptional $\rho \in \Irr(W)$
(up to $\approx$) never coincide, hence $\HH(1)' \simeq \sl_2^2 \times \sl_3^3  
\times \sl_6^6 \times \sl_8^3 \times \sl_9^2$ is the kernel of
the representation $U_{j,j^2} \oplus U'_{j^2,j}$ restricted to $\HH'$.
We get that $Z(\HH(1)) = \C T_{\SS}(1,1)$ by the same argument as above.

The situation for $\HH(-1)$ is a lot messier. As before, computing the
dimensions of $\rho(\HH(-1)')$ determines the type of the corresponding
ideals (notice that $\sl_2,\sl_3,\sl_2 \times \sl_3$
are the only semisimple algebras of dimensions 3,8, and 11 respectively) ;
computing the dimensions of the $(\rho^1 \oplus \rho^2)(\HH(-1)')$
determines the simple ideals of $\HH(-1)$. We get that
$$
\begin{array}{lcl}
\HH(-1) &= &\sl(A_1) \times \sl(A_2) \times \sl(B_1) \times \sl(B_2) \times
\sl(C) \times \sl(D) \times \sl(E) \\ &\simeq& \sl_2^2 \times \sl_3^2 \times \sl_6 \times \sl_8 \times
\sl_9 \\
\end{array}
$$
with $\dim A_i = 2$, $\dim B_i = 3$, $\dim C = 6$, $\dim D = 8$ and $\dim E = 9$.
Under this identification, the isomorphism type of the
representations $\rho_{\HH(-1)'}$ is then determined using
$\dim \rho$, the dimension of the invariants in
$\rho^1_{\HH(-1)'} \otimes \rho^2_{\HH(-1)'}$ (which distinguish e.g.
between $V$ and $V^*$ for a representation of $\sl(V)$),
and the dimensions of $(\rho^1_{\HH(-1)'} \otimes \rho^2_{\HH(-1)'})( \mathsf{U} \g)$,
which often determines the number of irreducible components in
$(\rho^1_{\HH(-1)'} \otimes \rho^2_{\HH(-1)'})$. The result is
tabulated below. In this table, $S^2 V$ denote the symmetric square of $V$, and $F_{[2,1]}$
denotes the Schur functor associated to the partition $[2,1]$, so that
$V^{\otimes 3} = S^3 V \oplus \Lambda^3 V \oplus 2 F_{[2,1]}(V)$.

$$
\begin{array}{|c||cccccc|}
\hline
\hline
\rho & U_{j,j^2} & U_{1,j} & U_{1,j^2} & V & &  \\
\hline
\rho(\HH(-1)') & \sl(V_{\rho}) & \sl(V_{\rho})& \sl(V_{\rho})& \sl_2 & & \\
\rho_{\HH(-1)'} & A_1  & A_2 & A_2 & S^2 A_2 & & \\
\hline
\hline
\rho & U'_{j,1} & U'_{1,j^2} & U'_{j^2,j} & U'_{j^2,1} & U'_{1,j} & 
   U'_{j,j^2}  \\
\hline
\rho(\HH(-1)') & \sl(V_{\rho})& \sl(V_{\rho})& \sl(V_{\rho})& \sl(V_{\rho})&\sl(V_{\rho}) &\sl(V_{\rho}) \\
\rho_{\HH(-1)'} & B_1 & B_1 & B_2 & B_1^* & B_1^* & B_2^*  \\
\hline
\hline
\rho & V_{j^2,1} & V_{j,j^2} & V_{1,j} & V_{j,1} & V_{j^2,j} & V_{1,j^2}  \\
\hline
\rho(\HH(-1)') & \sl_3 & \sl(V_{\rho})&\sl_3 \times \sl_2 & \sl_3 & \sl(V_{\rho}) & \sl_3 \times \sl_2  \\
\rho_{\HH(-1)'} &  S^2(B_1^*) & C & B_1 \otimes A_2 & S^2 B_1 & C^* & (B_1^*) \otimes A_2 \\
\hline
\hline
\rho & W_1 & W_{j^2} &W_j  & X & X^* &  \\
\hline
\rho(\HH(-1)') & \sl_3 & \sl(V_{\rho}) & \sl(V_{\rho}) & \sl(V_{\rho})& \sl(V_{\rho}) & \\
\rho_{\HH(-1)'} & F_{[2,1]}(B_1) & D & D^* & E & E^* & \\
%\hline
%\hline
%\rho & & & & & &  \\
%\hline
%\rho(\HH(-1)') & & & & & & \\
%\rho_{\HH(-1)'} & & & & & & \\
\hline
\end{array}
$$

\subsection{The spetsial Hecke algebra}

In view of the Brou\'e-Malle-Michel `Spetses' program, a specialization
of interest is when all the $\la_i$ are equal. We denote
$\HH_{st}$ the Lie subalgebra of $\HH$ generated by the
$\sum_{k=1}^{e_c-1} s^k$ for $s \in c \subset \SS$. We have
$\HH_0 \subset \HH_{st} \subset \HH$. The decomposition
of $\HH_{st}$ for the exceptional groups $G_4,G_{25}$ were done in
the previous section.
%, and shown there to be reductive. This is a
%general fact, explained by the following lemma.

%\begin{lemma} Let $\bula \in \AAA(W)$ such that,
%for all $c$ and $k$, $\la^c_k$ is the complex conjugate of $\la^c_{e_c-k}$.
%Then $\HH(\bula)$ is reductive.
%\end{lemma}
%\begin{proof}
%Let $(\ | \ )$ be the hermitian scalar product over $\C G$ defined
%by $(g|h) = \delta_{g,h}$ for $g,h \in G$, and embed
%$\C W$ in $\End_{\C}(\C W)$ by left multiplication. Then any $g \in W$
%acts unitarily on $\C W$, with adjoint given by $g^{-1}$. It follows
%that the $\sum_k \la^c_k s^k$ for $s \in \SS_+$ act by selfadjoint
%operators, hence the orthogonal of any stable subpace $U \subset \C W$
%is stable under $\HH(\bula)$. It follows that $\C W$ provides a
%faithful semisimple representation of $\HH(\bula)$, which is thus reductive.
%\end{proof}

We contend ourselves here to deal with the groups $G(d,1,r)$.

\begin{prop} \label{irrspets} Let $\bla = (\la_0,\dots,\la_{d-1})$ be a multipartition of $r$,
$W = G(d,1,r)$, and $\rho \in \Irr(W)$ the representation
of $W$ labelled by $\bla$. %Then $\rho_{\HH_{st}}$
%is semisimple and,
If $\la_0 \neq \emptyset$, then $\rho_{\HH_{st}}$
is irreducible. If $\la_0  = \emptyset$, then $\rho(\HH_{st}') = \rho(\HH_0')$.
\end{prop}
\begin{proof}
On note $P_i = (1/d) \sum_{k=0}^{d-1} \rho(t_i)^k$.
Letting $\bT = (T_0,\dots,T_{d-1})$
be a multitableau of shape $\bla$, we have $P_i \bT = 0$ if
$i \not\in T_0$ and $P_i \bT = \bT$ otherwise.
If $\la_0 = \emptyset$, then $\rho(\mathsf{U}\HH_{st}) = 
\rho(\mathsf{U}\HH_{0})$ hence $\rho_{\HH_{st}}$ is semisimple.

Now assume $\la_0 \neq \emptyset$ and $U$ be a subspace
setwise invariant under $\HH_{st}$.
Since $\HH_0 \subset \HH_{st}$
it is invariant under $W_0 = G(d,d,r)$, hence a sum of the
irreducible components. In particular, it is irreducible
unless the sequence $\la_0,\la_1,\dots$ has a period $0 < u < d$.
In that case, the irreducible components for $W_0$
are the eigenspaces of the endomorphism $S$ of order $d/u$
defined by $S(\bT) = (T_{u},T_{1+u},\dots,T_{d-1+u})$
(see \cite{MARINMICHEL} \S 2.4), hence $U$ is setwise
invariant $S$. Now let $v = \sum \alpha_{\bT} \bT \in U$
with $v \neq 0$, that is $\alpha_{\bT^0} \neq 0$ for some
multitableau $\bT^0$. Since $\la_0 \neq \emptyset$ there
exists some $i$ in $(\bT^0)_0$. Then $w = P_i v = \sum_{i \in
\bT_0} \alpha_{\bT} \bT \in U \setminus \{ 0 \}$.
By definition of $S$ the family $w,S(w),\dots,S^{d-1}(w) \in U$ is
free, hence the $\sum_{k=0}^{d/u-1} \zeta^{sk} S^k(w)$
for $\zeta = \exp(2 \ii \pi d/u)$ afford
eigenvectors for all eigenvalues of $S$.
It follows that $U$ meets every irreducible components
of the restriction to $W_0$, hence
$U$ is the whole space and $\rho_{\HH_{st}}$ is irreducible.
\end{proof}
%A consequence is that the sum of all irreducible representations
%of $W$ provides a faithful semisimple representation of $\HH_{st}$,
%which proves the following.
%\begin{cor} For $W = G(d,1,r)$, the Lie algebra $\HH_{st}$ is reductive.
%\end{cor}

%\begin{lemma} Let $\rho \in \Irr(G(d,1,r))$ associated
%to some multipartition $\bla$. Th

%\end{lemma}
We assume $r \geq 3$.

\begin{prop} Let $W = G(d,1,r)$ with $r \geq 3$, and $W_0 = G(d,d,r) < W$.
Let $\rho^1,\rho^2 \in \Irr(W)$ associated to multipartitions $\bla,\bmu$.
If both $\bla, \bmu$ have a single nonempty part, then $\rho^1_{\HH_{st}'} \simeq \rho^2_{\HH_{st}'}$
iff this single part is the same. Assuming this is not the case,
$\rho^1_{\HH_{st}'} \simeq \rho^2_{\HH_{st}'}$ iff
$\Res_{W_0} \rho^1 = \Res_{W_0} \rho^2$
%$\rho^1_{\HH_{0}'} \simeq \rho^2_{\HH_{0}'}$ 
when $\la_0 = \mu_0 = \emptyset$,
and $\rho^1_{\HH_{st}'} \simeq \rho^2_{\HH_{st}'}$ iff $\bla = \bmu$
otherwise.
%either $\forall i in [2,d-1] \la_i = 
%\emptyset$ or 
%$\la_0 = \emptyset$, then
%$\rho^1_{\HH_{st}'} \simeq \rho^2_{\HH_{st}'}$ if and only if
%$\Res_{W_0} \rho^1 = \Res_{W_0} \rho^2$
%$\rho^1_{\HH_0'} \simeq \rho^2_{\HH_0'}$ 
%and $\mu_0 = \emptyset$.
%Otherwise $\rho^1_{\HH_{st}'} \simeq \rho^2_{\HH_{st}'}$
%iff $\rho^1 = \rho^2$.
%If $\bla = (\la_0,\emptyset,\dots)$ with $\la_0 \neq \emptyset$,
%then $\rho^1_{\HH_{st}'} \simeq \rho^2_{\HH_{st}'}$ if and only if
%$\rho^1_{\HH_0'} \simeq \rho^2_{HH_0'}$ with $\m
\end{prop}

\begin{proof}
%Assume that $\rho^1$ is associated to the multipartition
%$\bla = (\la_0,\dots,\la_{d-1})$ and
We assume that $\rho^2_{\HH_{st}'} = \rho^1_{\HH_{st}'}$.
Since $\rho^1_{\HH_0'} \simeq \rho^2_{\HH_0'}$,
we know from \cite{IH2} that $\rho^2$ is associated to
a multipartition $\bmu = (\mu_0,\dots,\mu_{d-1})$
with $\mu_i = \la_{i+k}$ for all $i$ and a given $k \in [1,d-1]$.
We thus can assume that $V_{\rho^1} = V_{\rho^2}$ has
for basis the collection of multitableaux of shape $\bla$, with
$\rho^1$ affording the usual action on them,
$\rho^1(w) = \rho^2(w)$ for $w \in W_0$, and $\rho^2(t_i) \bT = 
\zeta^j \bT$ iff $\rho^1(t_i) \bT = \zeta^{j-k}$.
Now let $\mathbf{t}_i = \sum_{j=0}^{d-1} t_i^j$. Since
$\rho^1_{\HH_{st}'} \simeq \rho^2_{\HH_{st}'}$ we have $\rho^2(
\mathbf{t}_i) = \rho^1(\mathbf{t}_i) + \om_i$
for some $\om_i \in \C$. 
Since $\Sp(\rho^1(\mathbf{t}_i)),\Sp(\rho^2(\mathbf{t}_i))
\subset \{ 0 , d \}$ we get that, either
the $\rho^1(\mathbf{t}_i)$ are scalars, or $\om_i = 0$ for all $i$.
If the $\rho^1(\mathbf{t}_i)$ are not scalars, then
$\rho^1(\mathbf{t}_i) = \rho^2(\mathbf{t}_i)$
for all $i$, and also $\la_0 \neq \emptyset$ (otherwise $\rho^1(\mathbf{t}_i) = 0$).
Let then $\bT$ be a multitableau of shape $\bla$
with $1 \in T_0$. We have
$\rho^1(\mathbf{t}_1) \bT = d \bT$ and $\rho^2(\mathbf{t}_1) \bT = 0$
unless $k = 0$, which implies $\rho^1 = \rho^2$.

We now assume that the
$\rho^1(\mathbf{t}_i)$,
$\rho^2(\mathbf{t}_i)$ are scalars. Then
$\rho^1_{\HH_{st}'} \simeq \rho^2_{\HH_{st}'}$
iff $\rho^1_{\HH_{0}'} \simeq \rho^2_{\HH_{0}'}$, which is equivalent
to $\Res_{W_0}
\rho^1 = \Res_{W_0} \rho^2$ by \cite{IH2}.
This
case means that either $\bla, \bmu$ have a single part,
that is the $\rho^1(t_i)$ and $\rho^2(t_i)$ themselves are scalars,
and then $\rho^1_{\HH_0'} \simeq \rho^2_{\HH_0'}$
implies that these single parts are the same,
or $\la_0 = \mu_0 = \emptyset$. This concludes the proof,
as the converse implications are obvious.
\end{proof}

%, so we can assume
%that it is not the case. This implies that
%$\la_0 \neq \emptyset$ hence
%$\rho^1_{\HH_{st}}$ is irreducible by the previous proposition.
%then , so we can
%assume that $\rho^1_{\HH_{st}}$ is irreducible, and moreover that
%$\rho^1(t)$ is not a scalar.
%Since $\Sp(\rho^1(\mathbf{t}_i))
%,\Sp(\rho^2(\mathbf{t}_i)) \subset \R$, we have $\om_i \in \R$.
%Since $\rho^1(\mathbf{t}_i)$ is not a scalar, we have $\la_0 \neq
%\emptyset$. 

\begin{prop} \label{propdualgd1r} Let $W = G(d,1,r)$ with $r \geq 3$.
Let $\rho \in \Irr(W)$. If $\rho^2 \in \Irr(W)$ satisfies
$(\rho^2)_{\HH_{st}'} \simeq (\rho_{\HH_{st}'})^*$, then
$(\rho^2)_{\HH'} \simeq (\rho_{\HH'})^*$.
\end{prop}

\begin{proof}
Recall from \cite{IH2} that, if $\rho_0$ is the restriction
to $W_0 = G(d,d,r)$ of $\rho \in \Irr(W)$ associated to $\bla = (\la_0,\la_1,\dots,
\la_{d-1})$, then $\rho_0^* \otimes \eps$ is the restriction to
$W_0$ of the representation labelled $\bla^* := (\la_0',\la'_{d-1},\dots,\la'_1)$.

Assume $\rho,\rho^2 \in \Irr(W)$ with $(\rho_{\HH_{st}})^* \simeq
\rho^2_{\HH_{st}'}$. We have $(\rho_{\HH_{st}})^*_{|\HH_0'} \simeq
(\rho_{\HH_0'})^* \simeq ((\rho_{|W_0})^* \otimes \eps)_{\HH_0'}$.
It follows that $\mu_i = \la'_{k-i}$ for some $k \in [1,d-1]$,
by \cite{IH2}. We define a linear isomorphism $V_{\rho} \to V_{\rho^2}$
by $\sigma(\bT) = \bT' := (T'_{0+k},\dots,T'_{j+k},\dots,T'_{1+k})$,
where $T'$ denotes the transpose of the (standard) tableau $T$.
Transporting the representation $\rho^2$ to $V_{\rho^1}$ using
$\sigma$, this identify $\rho^2(x)$ with $- ^t \rho(x)$
for any $x \in \HH_0$. For $\bT$ a multitableau of
shape $\bla$, we have (under $\sigma$) $\rho^2(t_i) \bT = 
\zeta^j \bT$ iff $\rho^1(t_i) \bT = \zeta^{k-j} \bT$. Since $\rho^1(\mathbf{t}_i)$
and $\rho^2(\mathbf{t_i})$ have spectrum $\{ 0, d \}$, the
only possibility for $\rho^2(\mathbf{t}_i)$ to be conjugate to
$\om_i - ^t \rho^1(\mathbf{t}_i)$ for some scalar $\om_i$ is that $\om_i = d$, and
$\rho^1(\mathbf{t}_i) \bT = d \bT $ iff $\rho^2(\mathbf{t}_i) \bT = 0$.
Let $a \in [1,d-1]$ with $\la_a \neq \emptyset$, and choose
$\bT$ with $1 \in T_a$. Then $\rho^1(\mathbf{t}_1) \bT = 0$,
as $a \neq 0$, hence $\rho^2(\mathbf{t}_1) \bT = d$. This implies
$\rho^2(t_1) \bT = \bT = \zeta^0 \bT$, whence $\rho^1(t_1)\bT = \zeta^k \bT$,
$a \equiv k$ modulo $d$ and $a = k$. Then $\mu_k = \la'_{k-k} = \la'_0$
and $\mu_0 = \la'_{k-0} = \la'_k$. In that case we have $\rho^2 = (\rho^1)^*
\otimes \chi$ with $\chi_{|W_0} = \eps$ and $\chi(t_1) = \zeta^k$,
hence $\rho_{\HH'} \simeq (\rho^2)_{\HH'}$.
\end{proof}

\begin{prop}
Let $W = G(d,1,r)$ with $r \geq 3$ and $\rho \in \Irr(W)$
associated to a multipartition $\bla$ with at least 2 parts.
If $\la_0 = \emptyset$ then $\rho(\HH'_{st})  = \rho(\HH'_0)$,
otherwise $\rho(\HH'_{st}) = \rho(\HH'_s) = \rho(\HH')$.
\end{prop}
\begin{proof}
The case $\la_0 = \emptyset$ is clear, so we assume $\la_0 \neq \emptyset$.
In case $\rho \in \QRef(W)$ we can write $\rho(\mathbf{t}_i) = \alpha
\rho(t_i) + \beta$ for some $\alpha,\beta$ with $\alpha \neq 0$,
hence $\rho(\HH_s) \supset \sl(V_{\rho})$ implies $\rho(\HH_{st})
\subset \sl(V_{\rho})$ hence $\rho(\HH_{st}') = \sl(V_{\rho})$,
and this implies $\rho(\HH'_{st}) = \rho(\HH'_s)$ if
$\rho \in \LRef(W)$ and $\la_0 \neq \emptyset$.
The rest of the proof then follows verbatim the lines of
\S \ref{proofstructgdeer}, using the irreducibility of $\rho_{\HH'_{st}}$
proved by proposition \ref{irrspets}, and that $\rho_{\HH_{st}'}$
is selfdual if and only if $\rho_{\HH'}$ is so, by proposition \ref{propdualgd1r}.
The only change to make is in section \ref{soussection313},
in case $\bla$ has the form $(\la_0,\dots,\la_0,\dots)$
and $\rho$ factorizes through $G(d,d/2,r)$. But then $(2/d) \rho
(\mathbf{t}_i) -1 = \rho(t_i)$ and \cite{IH2} can also be applied,
as $\rho(\HH_{st})$ equals the image of the infinitesimal
Hecke algebra of $G(d,d/2,r)$ in the corresponding representation.
\end{proof}

\section{Relations between $\Ad(g)$ and $\ad(g)$}

\subsection{Preliminaries about cyclotomic fields}

We first need various preliminary results on cyclotomic fields.

\begin{lemma} \label{lemcumun} $\Q(\mu_n) = \Q(\mu_m)$ with $m \leq n$ if and only
if $n = 2m$ with $m$ odd.
\end{lemma}
\begin{proof} Since $\Q(\mu_n) \cap \Q(\mu_m) = \Q(\mu_{\gcd(m,n)})$
we can assume that $m$ divides $n$. Letting $\varphi$ denote the Euler
function, by taking Galois groups this yiels $\varphi(m) = \varphi(n)$.
Since, for $p$ prime, $\varphi(p^r) = p^{r-1}(p-1)$, we get $\varphi(m)=\varphi(n)$
implies $n = ma$ with $a$ prime to $m$ and $\varphi(a) = 1$, that is $a = 2$.
\end{proof}

We will use the following version of Goursat lemma.
\begin{lemma} \label{lemK2} Let $K$ denote a finite Galois extension of $\Q$ and $B$
a unital $\Q$-subalgebra of $K^2$ with $p_1(B) = p_2(B) = K$,
where $p_i$ is the $i$-th projection $K^2 \to K$. Then either
$B = K^2$ or there exists $\sigma \in \Gal(K|\Q)$ with $B = \{ (x,\sigma(x)) \ | \ x \in K \}$.
\end{lemma}
\begin{proof}
If $B$ is not an integral domain, then it contains
some nonzero and noninvertible element of $K^2$, that can chosen (by symmetry) of the
form $(0,x)$ for $x \neq 0$. Since $p_2$ is onto we habe $b \in B$ with $p_2(b) = x^{-1}$,
hence $(0,1) = (0,x)b \in B$ and similarly $(1,0) = 1 - (0,1) \in B$. Then
choosing for any $x,y \in K$ preimages $a,b \in B$ with $p_1(a) = x$ and $p_2(b) = y$
we get $(x,y) = a(1,0)+b(0,1) \in B$ and $B = K^2$.
Now assume that $B$ is an integral domain. Since $B$ has finite dimension over $\Q$
it is a field. Then the $(p_i)_{|B} : B \to K$ are injective, so they are isomorphisms.
Letting $\sigma = p_2 \circ (p_1)_{|B}^{-1} \in \Gal(K|\Q)$ we thus
get $B = \{ (x,\sigma(x) \ | \ x \in K \}$.
\end{proof}

\begin{lemma} \label{lemKbruhat} Let $B$ be a unital $\Q$-subalgebra of $A = \prod_{d|n} \Q(\mu_d)$
such that $p_d(B) = \Q(\mu_d)$ where $p_d = A \to \Q(\mu_d)$ is the
natural projection. We assume that, if $\Q(\mu_d) = \Q(\mu_{d'})$
there exists $b \in B$ such that $p_{d'}(b) \not\in \Gal(\Q(\mu_d)|\Q) p_d(b)$.
Then $B = A$.
\end{lemma}
\begin{proof}
If $\Q(\mu_d) = \Q(\mu_{d'})$ with $d,d'$ dividing $n$ we can
assume $d'=2d$ with $d$ odd. Then the projection of $B$ to
$\Q(\mu_d) \times \Q(\mu_{d'})$ is a unital $\Q$-subalgebra
that satisfies the assumptions of the previous lemma, hence 
$p_{d'}(b) \not\in \Gal(\Q(\mu_d)|\Q) p_d(b)$ implies that
it is equal to $\Q(\mu_d) \times \Q(\mu_{d'})$.
Let $D$ the set of divisors of $n$. For $I \subset D$, we denote
$p_I : A \to A_I = \prod_{d \in I} \Q(\mu_d)$ the natural projection,
and by contradiction we choose a minimal $I$ with $p_I(B) \neq A_I$.
The situation $\Q(\mu_{d}) = \Q(\mu_{d'})$ for all $d,d' \in I$
implies $|I| \leq 2$ by lemma \ref{lemcumun}, and is excluded
for $|I| = 1$ by the hypothesis, and for $|I| = 2$ by lemma
\ref{lemK2}. So there exists $d_1,d_2 \in I$ with
$\Q(\mu_{d_1} ) \neq \Q(\mu_{d_2})$. Since $\Q(\mu_d) \subset \Q(\mu_{\infty})$
is uniquely determined by its Galois group over $\Q$, this implies
$\Q(\mu_{d_1}) \not\simeq \Q(\mu_{d_2})$.

If $p_I(B)$ was an integral domain, it would be a field as $\dim_{Q} p_I(B) \leq
\dim B \leq \dim A < \infty$. Then $p_{d_1}$ and $p_{d_2}$ would induce
isomorphims $\Q(\mu_{d_1}) \simeq B \simeq \Q(\mu_{d_2})$, a contradiction.
So there exists $b \in B$ with $p_I(b) \neq 0$ and $p_d(b) = 0$
for some $d \in I$. Let $J = \{ d \in I \ | \ p_d(b) = 0 \} \neq \emptyset$.
We have $I \setminus J \neq \emptyset$ as $p_I(b) \neq 0$, and $p_{I \setminus
J}(B) = A_{I \setminus J}$, $p_J(B) = A_J$ by the minimality assumption.

Since $p_{J \setminus I}(b)$ is invertible in $A_{I \setminus J}$
it follows that there exists $c \in B$ with $p_{J \setminus I}(cb) =1$.
Since $p_I(b) = 0$ we have $p_I(cb) = 0$. Then, choosing for any $(x,y) \in p_J(B) \times
p_{I \setminus J}(B)$ preimages $u,v \in B$
with $x = p_J(u), y = p_{I \setminus J}(v)$,
we have $(x,y) = p_I(u(1-cb)+v cb)$ with $ u(1-cb)+v cb \in  B$, hence
$p_I(B) = p_J(B) \times
p_{I \setminus J}(B) = A_I \times A_{J\setminus I} = A_I$, a contradiction.
It follows that, for $I \subset D$ we have $p_I(B) = A_I$
and in particular $B = A_D = A$.

\end{proof}

In our situation, we will use the following Galois-theoretic lemma.

\begin{lemma} \label{lemgalzeta} Let $\zeta$ be a primitive $d$-root of $1$ and $a \geq 1$.
Then $\Q(\zeta)$ is generated as a unital $\Q$-algebra by $u = (\zeta-1)^{da}$
if $d,a$ are odd. Otherwise it is generated by $u$ and,
\begin{itemize}
\item if $d \equiv 0 \mod 4$, by $(\zeta^{d/2} -1)(\zeta^{d/4} - 1)^2 (\zeta^{d/2} -1)^{2(a-1)} \in \ii \Q^{\times}$;
\item if $d \equiv 2 \mod 4$, by $(\zeta^{d/2} - 1)(\zeta-1)^{d/2} (\zeta^{d/2} -1)^{2(a-1)} \in \ii \Q^{\times}$;
\item if $d$ is odd and $a$ even, by $(\zeta^a - 1)^d$.
\end{itemize}
\end{lemma}
\begin{proof}

 Since $\Gal(\Q(\zeta)|\Q)$ is abelian, any intermediate
extension of $\Q$ is Galois. Let $K$ be the one generated by the elements
in the statement, depending on $d$ and $a$. We need to show that
$\sigma \in \Gal(\Q(\zeta)|K)$ implies $\sigma = 1$. By contradiction
we assume $\sigma \neq 1$. Since $\sigma \in
\Gal(\Q(\zeta)|\Q)$ we have $\sigma(\zeta) = \zeta^{\alpha}$
for some $1 \leq \alpha \leq d-1$ prime to $d$.

For any $k \geq 1$, $(\zeta-1)^k$ is invariant under $\sigma$
iff $(\zeta^{\alpha}-1)^k = (\zeta-1)^k$. We can assume
$\zeta = e^{2 \ii \pi/d}$, and then $\zeta^{\alpha} -1 = 
e^{\alpha \ii \pi/d} (2 \ii) \sin(\alpha \pi/d)$
hence $|\zeta^{\alpha}-1|^k = |\zeta-1|^k$
iff $\sin(\alpha \pi/d) = \sin(\pi/d)$
iff $\alpha = 1$ or $\alpha = d-1$. For
$\alpha = d-1$, $(\zeta^{\alpha}-1)^k = (\zeta-1)^k$
can be written $e^{k(d-1) \ii \pi/d} = e^{k \ii \pi/d}$
that is $k(2-d) \in 2d \Z$.

Letting $k = da$, it follows that $\sigma(u) = u$
with $\sigma \neq 1$ implies that $\sigma$ is the complex
conjugation. If $da$ is odd, then so are $d$ and $d-2$,
hence $da(2-d) \not\in 2d \Z$ and a contradiction.
If $d$ is odd and $a$ even, then $\zeta^a$ is a primitive
$d$-th root of 1, and we get from the previous
argument that $\sigma$ fixes $(\zeta^a-1)^d$
iff $\sigma = 1$. We thus assume that $d$ is even,
and $\sigma$ the complex conjugation. Then $K$
is generated by some $(\zeta-1)^k$
and any element in $\ii \Q^{\times} \cap K$,
hence $\sigma = 1$, a contradiction that concludes the
proof.
\end{proof}

\subsection{On the algebra $\Q[X,Y]/(X^n-1,Y^n-1)$}

\begin{prop} \label{propQXY} Let $n \geq 2$. Then $XY^{n-1}$ belongs to the
unital $\Q$-subalgebra of $\Q[X,Y]/(X^n-1,Y^n-1)$
generated by the elements $X^k - Y^k$ for $1 \leq k \leq n$.
It also belongs to the subalgebra generated by
$X-Y$ if $n$ is odd or $n=2$.
\end{prop}
\begin{proof}
Let $C= \Q[X,Y]/(X^n-1,Y^n-1)$ and $C_0$ the subspace spanned (over $\Q$)
by the monomials in $X,Y$ of total degree $0$ modulo $n$. It has dimension
$n$ and is spanned by the elements $X^k Y^{-k}$ for $0 \leq k < n$.
It is also a subalgebra of $C$, generated by $Z = XY^{n-1} = XY^{-1}$,
since $Z^k = X^kY^{-k} = X^kY^{n-k}$. We need to show that $Z$ belongs
to the subalgebra $E$ of $C$ generated by the $X^k - Y^k$ and, for
$n$ odd or $n=2$, to the subalgebra $E_1$ of $E$ generated by $X - Y$.
The natural $\Q$-algebra morphism $\Q[T]/(T^n-1) \to C_0$
that maps $T$ to $Z$ is onto because $T^k$ is mapped to
$Z^k = X^k Y^{n-k}$ and is into by equality of dimensions, so we can identify
$C_0$ with $\Q[T]/(T^n-1) = \prod_{d|n} \Q[T]/\Phi_d(T) \simeq \prod_{d|n}
\Q(\mu_d) = A$, where $\Phi_d$ denotes the $d$-th cyclotomic polynomial,
and $\Q[T]/\Phi_d(T)$ is embedded in $\C$ through $T \mapsto e^{2 \ii \pi/d}$.
We let $p_d : A \to \Q(\mu_d)$ denote the natural projections. Then
$p_d(Z) = e^{2 \ii \pi/d}$ is a primitive $d$-th root of 1. Let $B$, $B_1$ denote the subalgebras
of $A$ corresponding to $E \cap C_0$ and $E_1 \cap C_0$, respectively.
For any $d | n$ and $a = n/d$, $B$ contains $(Z-1)^{da}, (Z^a-1)^d$ as well
as, for $d$ divisible by 4, $(Z^{d/2}-1)(Z^{d/4}-1)^2(Z^{d/2} -1)^{2(a-1)}$,
and, for $d$ even, $(Z^{d/2}-1)(Z-1)^{d/2}(Z^{d/2}-1)^{2(a-1)}$ ; 
$E_1$ contains $(Z-1)^{da}$.
%By lemma \ref{lemgalzeta} this implies $p_d(B) = \Q(\mu_d)$
%for all $d|n$, and $p_d(B_1) = \Q(\mu_d)$ for all $d|n$ if $n$ is odd.
We first assume $n$ odd. By lemma \ref{lemgalzeta}
we have $p_d(B_1) = \Q(\mu_d)$ for all $d|n$,
by lemma \ref{lemcumun} we have $\Q(\mu_d) = \Q(\mu_{d'}) \Rightarrow d=d'$
for $d,d' | n$, and by lemma \ref{lemKbruhat} this implies $B_1 = A$
(hence $B = A$, $C_0 \subset E_1$ and $XY^{n-1} \in E_1$.). The case
$n=2$ follows from $(X-Y)^2 = 2 - 2 XY$ in $C$.

Now assume $n$ is even. If there exists $d_1 < d_2$ with $d_1,d_2$
dividing $n$ such that $\Q(\mu_{d_1}) = \Q(\mu_{d_2})$, by lemma
\ref{lemcumun} we have $d_1$ odd and $d_2 = 2 d_1$. We let $\zeta_d = p_d(Z)= e^{2 \ii \pi/d}$.
Then $(\zeta_{d_1}-1)^n = (\zeta_{d_2}^2-1)^n$ is the
image of $(\zeta_{d_2}-1)^n$ by some $\sigma \in \Gal(\Q(\zeta_{d_2})|\Q)$
iff there exists $1 \leq \alpha \leq d_2-1$ prime to $d_2$ with
$(\zeta_{d_2}^2-1)^n = (\zeta_{d_2}^{\alpha}-1)^n$ hence
$\sin(\alpha \pi/d_2) = \sin(2 \pi/d_2)$,
that is $\alpha = 2$ or $\alpha = d_2 -2$.
Since $d_2 = 2d_1$ is even, this implies $\alpha$ even, contradicting
$\alpha$ prime to $d_2$. Since $(\zeta_{d_1} -1)^n$ and $(\zeta_{d_2} -1)^n$
are the images under $p_{d_1}$ and $p_{d_2}$, respectively, of $(Z-1)^n \in B$,
then $B \subset A$ satisfies the assumptions of lemma \ref{lemKbruhat}.
Then $B = A$, $C_0 \subset E$ and $XY^{n-1} \in E$.

\end{proof}

We now relate the endomorphisms $\Ad(g)$ and $\ad(g)$. Notice
that $\ad(g^k)$ commutes with $\ad(g^l)$ for any $k,l$.

\begin{theo} \label{polynome} Let $G$ be a finite group. For $g \in G$ we
let $\ad(g),\Ad(g) \in \End(\Q G)$ defined by $\ad(g) : x \mapsto gx-xg$
and $\Ad(g) : x \mapsto g x g^{-1}$. Let $n$ denote the order of $g \in G$.
Then
\begin{enumerate}
\item $\Ad(g)$ is a polynomial
in the $\ad(g^k), k \geq 1$ that depends only on $n$.
\item Let $n$ denote the order of $g \in G$. If $n$ is odd or $n=2$ then
$\Ad(g)$ is a polynomial in $\ad(g)$ that depends only on $n$.
\end{enumerate}
\end{theo}
\begin{proof}
Let $\Gamma \simeq \Z/n\Z$ denote the subgroup generated by $g$. The
algebra $\Q G$ is a $\Gamma$-bimodule, that is a $\Gamma \times \Gamma$-module.
It is thus enough to show that, for any \emph{complex} $\Gamma$-bimodule $M$,
then $\Ad(g)$ can be written as a rational polynomial in the $\ad(g^k)$,
or in $\ad(g)$ if $n$ is even, that depends only on $n$. We can take $M$ irreducible,
hence of dimension 1 and spanned by some nonzero $v \in M$,
for which $g.v = \zeta^r v$ and $v.g = \zeta^s v$, with $\zeta$
some fixed primitive $n$-th root of 1. Then $\ad(g^k)(v) = ((\zeta^r)^k - (\zeta^s)^k)v$
and $\Ad(g)(v) = (\zeta^r \zeta^{-s})v$. By proposition \ref{propQXY} we
get that $XY^{-1}$ is a rational polynomial in the $X^k - Y^k$,
and in $X-Y$ for $n$ odd or $n=2$, inside $\Q[X,Y]/(X^n-1,Y^n-1) = \Q (\Gamma
\times \Gamma)$, and the conclusion follows.
\end{proof}

For small $n$, the polynomials in the statement are easy to find :
$$
\begin{array}{lrcl}
n=2 &  \Ad(g) & = & 2 \Id - \ad(g)^2 \\
n=3 &  18\Ad(g) &=& 18 \Id  + 3 \ad(g)^3 + \ad(g)^6 \\
n = 4 &  8 \Ad(g) &=& 8 \Id - 3 \ad(g^2)^2 - \ad(g)^4 + 2 \ad(g)^2 \ad(g^2)\\
n=5 & 13750 \Ad(g) &=& 13750 \Id - 5875 \ad(g)^5 + 1900 \ad(g)^{10} \\
 & & & - 10 \ad(g)^{15} +
3 \ad(g)^{20} \\
n=6 & 183456 \Ad(g) & = & 183456 - 89573 \ad(g)^6 - 2210 \ad(g)^{12} \\
& & & + 55 \ad(g)^{18}-30576 \ad(g^2)^3+15288 \ad(g^2)^3 \ad(g^3)^2 \\
\end{array}
$$

We remark that it is not possible to express in general $\Ad(g)$ as a polynomial
in $\ad(g)$ for $n >2$ even. Indeed, in $\Q[X,Y]/(X^n - Y^n)$,
$XY^{-1}$ does not belong to the subalgebra generated by
the $(X-Y)^n$, as $(X-Y)^n$ is symmetric in $X,Y$ and $XY^{n-1}$ is not.
Since this algebra is equal to the intersection of the subspace
$C_0$ of homogeneous polynomials with total degree equal to $0$
modulo $n$ with the subalgebra generated by $X-Y$, this proves
that $XY^{n-1}$ is not a polynomial in $X-Y$. Now there exists groups
with a cyclic subgroup $\Gamma \simeq \Z/n\Z$,
e.g. $\Gamma \wr \mathfrak{S}_n \simeq G(n,1,n)$, which admit irreducible
representations whose restriction to $\Gamma$ contains all irreducible
representations of $\Gamma$. Taking for $g$ a generator of $\Gamma$
it follows that $\Q[X,Y]/(X^n-1,Y^n-1) = \Q (\Gamma \times \Gamma)$ embeds in $\Q G$,
hence $\Ad(g)$ is not a polynomial in $\ad(g)$ in these cases.

% \newpage
\tableofcontents

%\newpage

%\vfill
%\eject

\end{document}